\documentclass[12pt]{article}

\usepackage{amssymb}
\usepackage{amsmath}
\usepackage{amsthm}
\usepackage{hyperref}
\usepackage{mathrsfs}
\usepackage{bbm}
\usepackage{tikz-cd}
\usepackage[cal=boondox,scr=boondoxo]{mathalfa}
\usepackage{genyoungtabtikz}
\usepackage{ytableau}
\usepackage{xcolor}

\allowdisplaybreaks

{\theoremstyle{plain}%
 \newtheorem{theorem}{Theorem}

 \newtheorem{lemma}{Lemma}% 
}
{\theoremstyle{remark}
\newtheorem{remark}{Remark}
}
{\theoremstyle{definition}
\newtheorem{definition}{Definition}
\newtheorem{example}{Example}
}

\begin{document}

\begin{center}
 {\Large Kronecker coefficients via the Giambelli identity for Schur functions}

 \ 

 {\textsc{John M. Campbell}} 

 \vspace{0.1in}

 {\footnotesize Department of Mathematics and Statistics}

 {\footnotesize Dalhousie University}

 {\footnotesize Halifax, NS B3H 4R2}

 {\footnotesize Canada}

 \vspace{0.1in}

 {\footnotesize {\tt jh241966@dal.ca}}

 \vspace{0.1in}

\end{center}

\begin{abstract}
 One of the central open problems in both algebraic combinatorics and representation theory is to find a positive combinatorial rule for 
 Kronecker coefficients $ g_{\lambda \, \mu \, \nu}$. A notable advance in this direction is due to Blasiak, who proved a combinatorial 
 interpretation in terms of colored Yamanouchi tableaux for the case whereby one of the indexing partitions is hook-shaped. In this paper, 
 we introduce a framework for the evaluation and combinatorial interpretation of Kronecker coefficients, combining a Schur function 
 identity of Littlewood, the Giambelli identity for Schur functions, and Blasiak's combinatorial rule. This framework reduces the study of 
 Kronecker coefficients to alternating sums involving hook-indexed cases. As an application of this framework, we obtain combinatorial 
 interpretations of $g_{t, h^{(1)}, h^{(2)}}$ for two-row partitions $t$ and hook-like partitions $h^{(1)}$ and $h^{(2)}$ satisfying natural 
 conditions. More broadly, our approach provides a systematic method for extending hook-based combinatorial rules to wider families 
 of Kronecker coefficients.
\end{abstract}


\vspace{0.1in}


\section{Introduction}
 Let $\operatorname{GL}_{n}(\mathbb{C})$ denote the general linear group
 of order $n$ over $\mathbb{C}$. For integer partitions $\lambda$, $\mu$, 
 and $\nu$, \emph{Littlewood--Richardson (LR) coefficients} $c_{\mu \, \nu}^{\lambda}$ may be 
 defined via decompositions of (twofold) tensor products of irreducible polynomial $\operatorname{GL}_n(\mathbb{C})$-modules, with 
\begin{equation}\label{VtensorV}
 V^{\mu} \otimes V^{\nu} \cong \bigoplus_{\lambda} c_{\mu \, \nu}^{\lambda} V^{\lambda} 
\end{equation}
 holding in the stable range (i.e., for $n$ sufficiently large so that all partitions involved satisfy $\ell(\lambda), \ell(\mu), \ell(\nu) \le n$). 
 For partitions $\mu \vdash m$ and $\nu \vdash n$, the external tensor product $S^{\mu} \boxtimes S^{\nu}$ of Specht 
 modules $S^{\mu}$ and $S^{\nu}$ defines a representation of the direct product
 $\mathfrak{S}_m \times \mathfrak{S}_n$ of symmetric groups $\mathfrak{S}_{m}$ and $\mathfrak{S}_{n}$. 
 Inducing this representation to $\mathfrak{S}_{m+n}$ yields 
\begin{equation}\label{outerprod}
 \operatorname{Ind}_{\mathfrak{S}_{m} \times \mathfrak{S}_{n}}^{\mathfrak{S}_{m+n}}\big( S^{\mu} \boxtimes S^{\nu} \big) \cong \bigoplus_{\lambda } c_{\mu \, \nu}^{\lambda} S^{\lambda}.
\end{equation}
 The decompositions in \eqref{VtensorV} and \eqref{outerprod} correspond under the Frobenius characteristic map or, equivalently, via 
 Schur--Weyl duality. The structure constants arising in \eqref{VtensorV} and \eqref{outerprod} may equivalently be defined so that 
\begin{equation}\label{ssprod}
 s_{\mu} s_{\nu} = \sum_{\lambda} c_{\mu \, \nu}^{\lambda} s_{\lambda}, 
\end{equation}
 for the Schur basis $\{ s_{\lambda} \}_{\lambda}$ of the algebra $\textsf{Sym}$ of symmetric functions. \emph{Kronecker coefficients} 
 $g_{\lambda \, \mu \, \nu}$, which were introduced by Murnaghan in 1938 \cite{Murnaghan1938} (cf.\ \cite{Murnaghan1956}), 
 may be defined by analogy with \eqref{VtensorV} and via decompositions of (twofold) tensor products of Specht modules, with 
\begin{equation}\label{SpechtSpecht}
 S^{\mu} \otimes S^{\nu} \cong \bigoplus_{\lambda} g_{\lambda \, \mu \, \nu} S^{\lambda} 
\end{equation}
 for integer partitions $\mu$ and $\nu$ such that $|\mu| = |\nu|$. Writing $\operatorname{GL}_n = \operatorname{GL}_n(\mathbb{C})$, 
 we have, by direct analogy with \eqref{outerprod}, that Kronecker 
 coefficients arise as the multiplicities given by the restriction of irreducible representations from $\operatorname{GL}_{mn}$ to
 $\operatorname{GL}_{n} \times \operatorname{GL}_{m}$ (associated with the canonical morphism
 $\operatorname{GL}_{n} \times \operatorname{GL}_{m} \to \operatorname{GL}_{n m}$ whereby $(g, h) \mapsto g \otimes h$), 
 with $$ \operatorname{Res}^{\operatorname{GL}_{mn} }_{ \operatorname{GL}_n
 \times \operatorname{GL}_m }(V^{\lambda})
 \cong \bigoplus_{\mu, \nu} g_{\lambda \, \mu \, \nu} \big( V^{\mu} \otimes V^{\nu} \big). $$
 In a similar spirit, the appropriate analogue of \eqref{ssprod} involving Kronecker coefficients is such that 
\begin{equation}\label{sasts}
 s_{\lambda} \ast s_{\mu} = \sum_{\nu} g_{\lambda \, \mu \, \nu} s_{\nu} 
\end{equation}
 for the Kronecker/internal product $\ast$ on $\textsf{Sym}$, so that the left-hand side in \eqref{sasts} corresponds to the 
 Frobenius characteristic of the left-hand side of \eqref{SpechtSpecht}. 

 A combinatorial interpretation of $c_{\mu \, \nu}^{\lambda}$, i.e., in terms of \emph{Littlewood--Ri-chardson tableaux}, was developed 
 over much of the $20^{\text{th}}$ century \cite[pp.\ 176--177]{Sagan2001}, culminating with full proofs of this interpretation, 
 which traces back to a seminal 1934 paper from Littlewood and Richardson \cite{LittlewoodRichardson1934}, in the 
 1970s attributed to Thomas \cite{Thomas1974,Thomas1978} and to Sch\"utzenberger~\cite{Schutzenberger1977}. As suggested by 
 Lascoux \cite{Lascoux1980}, it is remarkable how the multiplicities arising in decompositions of \emph{outer} products on Specht
 modules, as in \eqref{outerprod}, have a well known interpretation with a long history, in contrast to how there is no known 
 combinatorial interpretation for the corresponding multiplicities associated with \emph{inner} products as in \eqref{SpechtSpecht}. 
 Indeed, the problem of determining such a combinatorial interpretation is, to this day, one of the most important open problems in 
 representation theory and in algebraic combinatorics. This problem is often referred to as the \emph{Kronecker problem}. 
 We introduce a framework for the evaluation and combinatorial interpretation of Kronecker coefficients, based on a combination of the 
 Giambelli identity \cite{Giambelli1903} (cf.\ \cite[pp.\ 47, 61]{Macdonald1995}), a compatibility identity of Littlewood 
 \cite{Littlewood1956} relating the coproduct and the Kronecker product on $\textsf{Sym}$, and Blasiak's combinatorial rule for
 hook shapes \cite{Blasiak20162018}. This framework provides a systematic reduction of Kronecker coefficients indexed by 
 general partitions to alternating sums involving hook-indexed Kronecker coefficients and Littlewood--Richardson coefficients. 
 Although we focus on specific families of partition shapes in this paper, the underlying reduction mechanism applies more broadly
 and suggests a general approach to extending hook-based combinatorial rules.

 The character version of the Specht module decomposition highlighted in \eqref{SpechtSpecht} is such that 
\begin{equation}\label{charver}
 \chi^{\mu} \chi^{\nu} = \sum_{\lambda} g_{\lambda \, \mu \, \nu} \chi^{\lambda}, 
\end{equation}
 and we find that \eqref{charver} is equivalent to 
\begin{equation}\label{gtochar}
 g_{\lambda \, \mu \, \nu} = \frac{1}{n!} \sum_{\sigma \in \mathfrak{S}_{n}} \chi^{\lambda}(\sigma) \chi^{\mu}(\sigma) \chi^{\nu}(\sigma). 
\end{equation}
 As emphasized by Lascoux \cite{Lascoux1980}, by writing a given Kronecker coefficient as a combination of irreducible characters of 
 $\mathfrak{S}_{n}$, i.e., according to the relation in \eqref{gtochar}, this approach is computationally inefficient and requires full 
 evaluations for all of the entries in the character table for $\mathfrak{S}_{n}$. Moreover, the expansion in \eqref{gtochar} does not seem to provide 
 much in the way of insight when it comes to positivity properties of or combinatorial interpretations of $g_{\lambda \, \mu \, \nu}$. 
 This highlights the interest in our Giambelli--Littlewood--Blasiak-based approach to the Kronecker problem.

\subsection{Outline}\label{secFrame}
 Let $\lambda$ be an integer partition written in Frobenius notation so that $\lambda = (\alpha_1, \alpha_2, \ldots, \alpha_d \, | \, \beta_1, 
 \beta_2, \ldots, \beta_d)$, for the number $d$ of diagonal boxes of $\lambda$, and where $\alpha_i$ and $\beta_i$, respectively, denote 
 the arm and leg length of the $i^{\text{th}}$ box, for $i \in \{ 1, 2, \ldots, d \}$. The Giambelli identity may then be formulated so that 
\begin{equation}\label{origGiambelli}
 s_{\lambda} = \det\left( s_{\left( \alpha_i + 1, 1^{\beta_j} \right) } \right)_{1 \leq i, j \leq d}. 
\end{equation} 
 Rewriting \eqref{origGiambelli} using the Leibniz formula, we find that 
\begin{equation}\label{rewriteGiambelli}
 s_{\lambda} = \sum_{\sigma \in \mathfrak{S}_{d}} \operatorname{sgn}(\sigma) \prod_{i=1}^{d} s_{\left( \alpha_i + 1, 1^{\beta_{\sigma(i)}} \right)}. 
\end{equation}

 The identity 
\begin{equation}\label{mainLittlewood}
 (s_{\lambda} s_{\mu}) \ast s_{\nu} = \sum_{ \substack{\tau \vdash |\lambda| \\ \eta \vdash |\mu|} } c_{\tau \, \eta}^{\nu} (s_{\tau} 
 \ast s_{\lambda}) (s_{\eta} \ast s_{\mu}) 
\end{equation}
 for partitions $\lambda$, $\mu$, and $\nu$ is originally due to Littlewood in 1956~\cite{Littlewood1956}. The identity 
 \eqref{mainLittlewood} can be interpreted as a compatibility relation $ (fg) \ast h = \sum (f \ast h_{(1)}) (g \ast h_{(2)})$ between the 
 Kronecker product and the coproduct $\Delta$ on $\textsf{Sym}$, using Sweedler notation and writing $\Delta(h) = \sum h_{(1)} 
 \otimes h_{(2)}$. A key to our construction in Section \ref{secNear} is given by applications of the Littlewood rule in 
 \eqref{mainLittlewood}. More precisely, we rewrite a Schur function using the Giambelli--Leibniz expansion in 
 \eqref{rewriteGiambelli} and then apply the Littlewood identity in \eqref{mainLittlewood} to expand the Kronecker product with a 
 fixed Schur function. Extracting coefficients yields an expression of a Kronecker coefficient (in which one of the indexing partitions 
 is a non-hook) as an alternating sum involving Littlewood--Richardson coefficients and hook-indexed Kronecker coefficients. The 
 latter admit combinatorial interpretations via Blasiak's rule \cite{Blasiak20162018}, and we then reduce the resulting alternating sum in a 
 cancellation-free manner to obtain new combinatorial formulas.

 We define a \emph{near-hook} as an integer partition of the form $(a, b, 1^c)$ for $a \geq b$, providing a natural extension of 
 hook-partitions. For the case of Schur functions indexed by near-hooks, the Giambelli--Leibniz expansion gives us that 
\begin{equation}\label{GLgives}
 s_{(a, b, 1^c)} = s_{(a, 1^{c+1})} s_{(b-1)} - s_{(a)} s_{(b-1, 1^{c+1})}. 
\end{equation}
 Applying the above steps (as made more explicit in Section \ref{secNear}), we obtain an alternating sum of the form 
\begin{equation}\label{nonexplicitsum}
 g_{\lambda \, (a, b, 1^{c}) \, \beta} = \sum_{ \eta, \delta, \theta } 
 c_{\eta \, \delta}^{\beta} g_{\theta \, (a, 1^{c+1}) \, \eta } c^{\lambda}_{\theta \, \delta} 
 - \sum_{ \eta, \delta, \theta } 
 c_{\eta \, \delta}^{\beta} g_{\theta \, (b-1, 1^{c+1}) \, \delta } c^{\lambda}_{\eta \, \theta}, 
\end{equation}
 and an explicit version of this expansion is highlighted in Theorem \ref{fundamentaltheorem}. Since both of the expressions 
 $g_{\theta \, (a, 1^{c+1}) \, \eta } $ and $g_{\theta \, (b-1, 1^{c+1}) \, \delta }$ involved within the summands in \eqref{nonexplicitsum} 
 admit combinatorial interpretations due to Blasiak \cite{Blasiak20162018}, the determination of a combinatorial interpretation for the 
 left-hand side of \eqref{nonexplicitsum}, for a given or restricted shape $\lambda$, reduces to the determination of cancellations 
 that can be formulated via bijections defined using
 LR/Blasiak tableaux. 

 For the purposes of this paper, we restrict the application of our technique to special cases of \eqref{nonexplicitsum}. However, the 
 above framework, relying on a threefold application of the Giambelli--Leibniz expansion, Littlewood's rule, and Blasiak's combinatorial 
 interpretation, can be applied much more broadly. Conceptually, this framework reduces the Kronecker problem to the analysis of 
 cancellation phenomena in structured families of Littlewood--Richardson and Blasiak tableaux.

 One might think that sums as in \eqref{nonexplicitsum} would be intractable in terms of requiring large or infeasible amounts of 
 cancellation, in order to determine explicit and cancellation-free formulas from \eqref{nonexplicitsum}, especially since these kinds of 
 problems often arise in algebraic combinatorics (as explored, notably, by Benedetti and Sagan \cite{BenedettiSagan2017} in a different 
 context). However, by restricting the superscripting partition of an LR coefficient, sums as in \eqref{nonexplicitsum} can be shown to 
 reduce/``collapse'' in remarkable ways, in regard to the nonzero terms that remain. 
 The remarkable ways in which sums as in \eqref{nonexplicitsum} simplify illustrate the 
 versatility and usefulness of our Giambelli--Littlewood--Blasiak-based approach. 

\subsection{Organization}
 We provide an outline of our new framework for evaluating and combinatorially interpreting Kronecker coefficients in Section 
 \ref{secFrame} (referring to Section \ref{secprelim} for preliminaries on notation and terminology). This is followed by Section 
 \ref{secBack}, which gives further background material, in addition to the background material covered above. 
 Our main results are given in Sections \ref{secNear} and \ref{secinter}. 

\section{Background and context}\label{secBack}

\subsection{Survey}\label{secSurvey}
 The following selection highlights families of partition shapes for which Kronecker coefficients admit explicit evaluations or combinatorial 
 interpretations, emphasizing those most closely related to our framework.

 \ 

\noindent \textbf{Hooks:} $ g_{h^{(1)} \, h^{(2)} \, \nu} $ for hooks $h^{(1)}$ and $h^{(2)}$ 
 \cite{GarsiaRemmel1985,Lascoux1980,Remmel1989,Rosas2001}; $ g_{t \, h \, \nu} $ for a two-rowed partition $t$ and a hook $h$ 
 \cite{BowmanDeVisscherOrellana2015,Remmel1992}; $ g_{(2^{a}, 1^{b}) \, h \, \nu} $ for
 a hook $h$ \cite{Remmel1992}; 
 $g_{\lambda \, h \, \mu }$ for a hook $h$ \cite{Blasiak20162018}. 

 \ 

\noindent \textbf{Two-row and near-two-row shapes:} $ g_{t^{(1)} \, t^{(2)} \, \nu } $ for two-rowed partitions $t^{(1)}$ and $t^{(2)}$ 
 \cite{BriandOrellanaRosas2009,RemmelWhitehead1994,Rosas2001}; $g_{(a, b) \, \mu \, \nu }$ for $\mu_1 - \mu_2 \geq 2 
 b$ \cite{BallantineOrellana2005,BallantineOrellana200507}; $g_{(a, a) \, (a+b,a-b) \, \nu }$ \cite{BrownvanWilligenburgZabrocki2010}; 
 $g_{(a, a) \, (a, a) \, \nu}$ \cite{GarsiaWallachXinZabrocki2012}; $g_{ (n, n-1, 1) \, (n, n) \, \nu }$ for $n \geq 2$ \cite{Tewari2015,TewariMSc}; 
$g_{ (n-1, n-1, 1) \, (n, n-1) \, \nu }$ for $n \geq 2$ \cite{Tewari2015,TewariMSc}; 
$g_{ (n-1, n-1, 2) \, (n, n) \, \nu }$ for $n \geq 3$ \cite{Tewari2015,TewariMSc}; 
$g_{ (n-1, n-1, 1,1) \, (n, n) \, \nu }$ for $n \geq 2$ \cite{Tewari2015,TewariMSc}; 
$g_{ (n, n, 1) \, (n, n, 1) \, \nu }$ for $n \geq 2$ \cite{Tewari2015,TewariMSc}. 

 \ 

\noindent \textbf{Rectangular and near-rectangular shapes:} $ g_{(n-1, 1) \, r \, \nu } $ for a rectangle $r$ \cite{BessenrodtKleshchev1999}; 
 $g_{r \, r \, \nu}$ for rectangular partitions $r$\footnote{See also \url{https://arxiv.org/abs/2511.02312}.} \cite{Manivel2011}; $ g_{r \, r \, 
 t } $ for rectangular partitions $r$ and for a partition $t$ with at most 
 two rows \cite{Vallejo2014}; 
$g_{(a^b) \, (a b - c, 1^c) \, \nu}$ for $c > d+1$ \cite{BallantineHallahan2016}; 
$g_{r^{(1)} \, r^{(2)} \, r^{(3)} }$ for certain rectangular partitions $r^{(1)}$, 
 $r^{(2)}$, and $r^{(3)}$ \cite{AmanovYeliussizov2023}; 
$g_{a^a \, a^a \, \nu}$ \cite{Zhao2024}. 

 \ 

\noindent The above catalogue emphasizes how there is no general mechanism to extend hook formulas to broader classes. Our 
 Giambelli--Littlewood--Blasiak-based approach fills this gap. This, combined with the work of Rosas \cite{Rosas2001}, yields explicit 
 combinatorial formulas for $g_{t, n^{(1)}, n^{(2)}}$ for naturally occurring families of two-row partitions $t$ and hook-like partitions 
 $n^{(1)}$ and $n^{(2)}$. 

\subsection{Kronecker coefficients and Giambelli's identity}
 The Giambelli identity, as given in Section \ref{secFrame} below, has previously been applied in the study of Kronecker coefficients, but in 
 a different way compared to our methods and results. In the work of Pak et al.\ \cite{PakPanovaVallejo2016} (cf.\ \cite{PakPanova2017}), 
 the Giambelli determinantal identity is used to rewrite Schur functions in terms of hook-shaped components, which are applied in the 
 study of character values and associated asymptotic properties using the Murnaghan--Nakayama rule. While Pak et al.\ use the 
 Giambelli identity as a tool for facilitating character computations, our work uses Giambelli’s formula in a more structural and 
 fundamental way to obtain explicit expansions of Kronecker coefficients themselves. Our applications of Giambelli's identity may also be 
 seen in relation to E\u gecio\u glu and Remmel's combinatorial proof of this identity \cite{EgeciogluRemmel1988}, as their 
 combinatorial proof helps to explain how the structure of the Giambelli expansions can be applied in relation to our goals. 

\subsection{Kronecker coefficients and Littlewood's rule}
 What we refer to as \emph{Littlewood's rule}, as given in \eqref{mainLittlewood}, naturally arises in the context of Kronecker coefficients, 
 as it provides an expansion for mixed products of Schur functions involving the Kronecker and usual products on $\textsf{Sym}$. Our 
 work may be viewed as complementary to that of Tewari \cite{Tewari2015}, who studies near-rectangular shapes, whereas we focus on 
 near-hook shapes. In both settings, Littlewood’s identity in \eqref{mainLittlewood} is used to derive and reduce alternating sums of 
 Schur functions, although the combinatorial structures involved differ significantly.

 There are many further and notable works on Kronecker coefficients involving the manipulation of sums derived from Littlewood's rule. 
 For example, Pak and Panova's paper \cite{PakPanova2014} on unimodality properties related to Kronecker products connects with our 
 work in terms of how the Littlewood identity in \eqref{mainLittlewood} is key to the derivations of their main results, i.e., via 
 manipulations of sums obtained from Littlewood's rule, producing alternating sums used, in turn, to obtain unimodality properties. 
 This reflects how we apply Littlewood's rule via alternating sums as in \eqref{nonexplicitsum}. 

\subsection{Blasiak's combinatorial interpretation}
 Mixed insertion, as introduced by Haiman \cite{Haiman1989}, generalizes Schensted insertion and may be seen as governing the 
 combinatorics of shifted tableaux and Hall--Littlewood functions at $t = -1$. 
 Haiman's introduction of mixed insertion laid the foundations for the work of Blasiak 
 \cite{Blasiak20162018}, which is central to our work. Our new applications of Blasiak's combinatorial interpretation of $g_{\lambda 
 \, h \, \mu}$ for a hook $h$ may also be seen in relation to how Haiman's work builds upon classical versions of \emph{jeu-de-taquin}. 
 Based on extant research concerning Blasiak's combinatorial interpretation, 
 our approaches (via Giambelli's identity, according to Theorem \ref{fundamentaltheorem}), as below, toward the application of 
 this combinatorial interpretation have not been considered previously in any meaningful way. 
 Full preliminaries on Blasiak's interpretation are given in Section \ref{secBlasiakinter}. 

\subsection{Applications related to Kronecker coefficients}
 Kronecker coefficients are considered as providing a cornerstone of the discipline of geometric complexity theory 
 \cite{BowmanDeVisscherEnyang2021}. As Bowman et al.\ also noted, Kronecker coefficients are also closely linked with properties 
 of both the spectra of quantum states \cite{ChristandlHarrowMitchison2007,ChristandlMitchison2006} and quantum entropies 
 \cite{ChristandlSahinogluWalter2018}. The problem of combinatorially interpreting Kronecker coefficients in relation to quantum systems
 has also been considered by Brown et al.\ \cite{BrownvanWilligenburgZabrocki2010}, who described how such interpretations are useful 
 in relation to computations involved in physics-based problems concerning quantum-mechanical systems known as \emph{qubits}. 
 More broadly, since Kronecker coefficients arise via the decomposition of tensor products of irreducible representations of symmetric 
 groups, such coefficients naturally arise in settings that involve both symmetry/permutations and tensoring. This may be considered in 
 relation to how Kronecker products arise in areas such as the geometry of flag varieties and invariant theory. 

 In addition to the references cited above, there has been a variety of different research developments, over recent years, on Kronecker 
 coefficients and their uses in relation to many different areas and topics in mathematics. Such research works have, for example, 
 concerned the decomposition of generalized Foulkes modules \cite{HegedusMadireddi2025} and log-concavity properties of 
 Kronecker coefficients \cite{Gui2025}. Additional progress concerns the Newton polytope of Kronecker products of Schur functions 
 \cite{PanovaZhao2024,PanovaZhao2025}, analogues of Kronecker coefficients for symmetric inverse semigroups 
 \cite{MazorchukSrivastava2025}, how the problem of evaluating Kronecker coefficients is equivalent to the problem of evaluating 
 reduced Kronecker coefficients~\cite{IkenmeyerPanova2024,IkenmeyerPanova2024Sem}, and the study of Kronecker coefficients via the 
 use of signed sums of vector partition function evaluations \cite{MishnaTrandafir2024}. 

\section{Preliminaries}\label{secprelim}
 For brevity, we assume familiarity with basic definitions associated with combinatorial objects such as integer partitions, Young diagrams 
 and partition tableaux, skew diagrams, etc., referring to many of the cited references below for details. 

\subsection{The self-dual Hopf algebra structure on $\textsf{Sym}$}
 The algebra $\textsf{Sym}$ of symmetric functions may be defined, for our purposes, as the free $\mathbb{Q}$-algebra with one 
 generator $h_{n}$ in each degree $n \in \mathbb{N}$, writing 
\begin{equation}\label{displaySym}
 \textsf{Sym} = \mathbb{Q}[h_1, h_2, \ldots] 
\end{equation}
 and referring to Macdonald's text for background \cite{Macdonald1995} and 
 for a more detailed construction of $\textsf{Sym}$, and for 
 details on aspects about $\textsf{Sym}$ not directly used in our work. The definition in \eqref{displaySym} gives rise to the notational 
 shorthand whereby $h_{\lambda} = h_{\lambda_{1}} h_{\lambda_{2}} \cdots h_{\lambda_{\ell(\lambda)}}$ for an integer partition 
 $\lambda$, giving rise to the \emph{complete homogeneous basis} $\{ h_{\lambda} \}_{\lambda}$ of $\textsf{Sym}$, letting this basis 
 be indexed by the family of integer partitions, and writing $h_{()} = h_{0} = 1$ for the unique empty partition $()$. Again from the 
 definition in \eqref{displaySym}, we find that the product $h_{\lambda} h_{\mu}$ of two complete homogeneous basis elements is 
 the $h$-basis element $h_{\nu}$ indexed by the partition $\nu$ obtained by sorting the concatenation of $\lambda$ and $\mu$. By 
 writing $\textsf{Sym}_{n}$ in place of the $\mathbb{Q}$-space spanned by $\{ h_{\lambda} : \lambda \vdash n \}$, we thus obtain the 
 graded algebra decomposition 
\begin{equation}\label{Symgraded}
 \textsf{Sym} = \bigoplus_{n=0}^{\infty} \textsf{Sym}_{n}. 
\end{equation}

 A key tool used in our work is the above referenced Giambelli identity, which expresses Schur functions determinantally, in a way that is 
 closely related to the \emph{Jacobi--Trudi rule} 
\begin{equation}\label{displayJT}
 s_{\lambda} = \det\left( h_{\lambda_i - i +j} \right)_{1 \leq i, j \leq \ell(\lambda)}, 
\end{equation}
 with the understanding that $h_{-n}$ vanishes for $-n < 0$. The relation in \eqref{displayJT} allows us to define the \emph{Schur basis}
 $\{ s_{\lambda} \}_{\lambda}$, which is typically regarded as \emph{the} basis of $\textsf{Sym}$. 

 The algebra $\textsf{Sym}$ is typically endowed with a bialgebra structure and a Hopf algebra structure, with a coproduct 
 $\Delta\colon \textsf{Sym} \to \textsf{Sym} \otimes \textsf{Sym}$ that may be defined so that 
\begin{equation}\label{Deltah}
 \Delta h_{n} = h_{0} \otimes h_{n} + h_{1} \otimes h_{n-1} + \cdots + h_{n} \otimes h_{0}, 
\end{equation}
 with \eqref{Deltah} being extended so that $\Delta$ is compatible with the multiplicative operation on $\textsf{Sym}$. 

 As in the work of Rosas \cite{Rosas2001} (which is heavily used in Section \ref{twonearhook} below), our construction heavily relies on 
 the use of indicator functions, writing $$ \text{{\bf 1}}_{P} = \begin{cases} 1, & \text{if $P$ is true;} \\
 0, & \text{otherwise,} \end{cases} $$ for a given proposition $P$ (thus generalizing the Kronecker delta function). This leads us to define 
 the \emph{Hall scalar product} or \emph{Hall inner product} $\langle \cdot, \cdot \rangle$ 
 according to the relation 
\begin{equation}\label{Schurdual}
 \langle s_{\lambda}, s_{\mu} \rangle = \text{{\bf 1}}_{\lambda = \mu}, 
\end{equation}
 with \eqref{Schurdual} giving us the 
 orthonormality of the Schur basis. 
 The self-duality relation in \eqref{Schurdual} together with the product rule in 
 \eqref{ssprod} yield the coproduct rule such that 
 $ \Delta s_{\lambda} = \sum_{\mu, \nu} c_{\mu \, \nu}^{\lambda} s_{\mu} \otimes s_{\nu}$. 

 By endowing, as above, the graded algebra structure in \eqref{Symgraded} with the coproduct defined via \eqref{Deltah}, this gives
 $\textsf{Sym}$ the structure of both a bialgebra and a Hopf algebra. The interested reader may consult a number of standard, general 
 textbook references on Hopf algebras \cite{Abe1980,CartierPatras2021,Montgomery1993,Sweedler1969}. A paradigm shift given by 
 connecting the discipline of combinatorics with properties of coalgebras and bialgebras is given by the work of Joni and Rota 
 \cite{JoniRota1979}, and this helped to lay the foundation for the introduction of the concept of a \emph{Combinatorial Hopf 
 Algebra} (CHA) \cite{AguiarBergeronSottile2006}, with CHA structures on symmetric functions and generalizations being central 
 in the study and application of CHAs, which forms a major part of modern-day algebraic combinatorics. 

 A crucial component to our technique below is related to how the multiplicative operation on $\textsf{Sym}$ relates to the operations
 $\ast$ and $\Delta$ on $\textsf{Sym}$, according to the Littlewood relation in \eqref{mainLittlewood}. For further background material 
 on the Hopf algebra $\textsf{Sym}$, in addition to Macdonald's text \cite{Macdonald1995}, one may also consider the appropriate 
 text by Aguiar and Mahajan \cite{AguiarMahajan2010}. 

\subsection{Littlewood--Richardson coefficients}
 Recall the definition for LR coefficients given in \eqref{ssprod}. 

\begin{example}
 We adopt the convention whereby tuples of single-digit integers may be written as words (i.e., in the monoid $\{ 0, 1, \ldots, 
 9 \}^{\ast}$). From the expansion 
\begin{multline*}
 s_{321} s_{21} = s_{32211} + s_{3222} + s_{33111} + 2 s_{3321} + s_{333} + s_{42111} + \\ 
 2 s_{4221} + 2 s_{4311} + 2 s_{432} + s_{441} + s_{5211} + s_{522} + s_{531}, 
\end{multline*}
 we can see, for example, that $ c_{321 \, 21}^{4311} = 2$. 
\end{example}

 The standard combinatorial interpretation for LR coefficients requires, for our purposes, the following definitions. 

\begin{definition}\label{defineSSYT}
 A \emph{semistandard Young tableau (SSYT)} is a tableau of skew shape $\lambda / \mu$ labeled with natural numbers such that the 
 labels are weakly increasing along the rows and such that the labels are strictly increasing down the columns. 
\end{definition} 

\begin{example}
 The tableau
\begin{equation}\label{firstLRtab}
 \begin{ytableau}
 \none & \none & \none & \none & 1 \\
\none & \none & 1 & 1 \\
1 & 2 \\ 
 3 \\ 
\end{ytableau}
\end{equation}
 satisfies the given conditions in Definition \ref{defineSSYT}, with a skew shape of $5421/42$. 
\end{example}

\begin{definition}
 Let $T$ be a SSYT with $m$ as its maximal label. Then the \emph{weight} associated with $T$ is the tuple $(t_1, t_2, \ldots, t_m)$ such 
 that $t_i$ is the number of labels equal to $i$ appearing in $T$ for $i \in \{ 1, 2, \ldots, m \}$. 
\end{definition}

\begin{example}
 The weight of the SSYT displayed in \eqref{firstLRtab} is $411$. 
\end{example}

\begin{definition}
 The \emph{reading word} associated with a SSYT $T$ is formed by reading the labels of $T$ row by row, from right to left, starting with 
 the top-right label and proceeding downwards. 
\end{definition}

\begin{example}\label{exreadword}
 The reading word for \eqref{firstLRtab} is $111213$.
 \end{example}

\begin{definition}\label{defineYama}
 A word $w$ is said to satisfy the \emph{Yamanouchi condition} if, for each initial subword of $w$, 
 the number of $i$'s is at least the number of $(i+1)$'s, for each $i \geq 1$.
\end{definition}

\begin{example}
 We find that the reading word evaluated in Example \ref{exreadword} satisfies the condition in Definition \ref{defineYama}. For example, 
 the prefix $1112$ of this reading word is such that the number of $1$'s is at least the number of $2$'s. 
\end{example}

 A standard combinatorial interpretation of $c_{\mu \, \nu}^{\lambda}$ is such that $ c_{\mu \, \nu}^{\lambda}$ is the number of SSYTs 
 of shape $\lambda / \mu$ with weight $\nu$ and with a reading word satisfying the Yamanouchi condition. For the purposes of this 
 paper, a \emph{Littlewood--Richardson tableau} is a tableau of this form (cf.\ 
 \cite[p.\ 143]{Macdonald1995} \cite[p.\ 177]{Sagan2001}). 

\begin{example}
 We find that $ c^{5421}_{42 \, 411} = 2$, and that the associated tableaux are as in \eqref{firstLRtab} and 
\begin{equation*}
 \begin{ytableau}
 \none & \none & \none & \none & 1 \\
\none & \none & 1 & 2 \\
1 & 1 \\ 
 3 \\ 
\end{ytableau}.
\end{equation*}
 In contrast, the tableau 
\begin{equation*}
 \begin{ytableau}
 \none & \none & \none & \none & 1 \\
\none & \none & 1 & 1 \\
1 & 3 \\ 
 2 \\ 
\end{ytableau}
\end{equation*}
 is not an LR tableau, as the Yamanouchi condition does not hold. 
\end{example}

\subsection{Kronecker products/coefficients}
 Kronecker coefficients may be defined, as above, according to the decomposition in \eqref{SpechtSpecht}. The implementation of 
 the Kronecker product operation in the {\tt SageMath} system thus provides, via the relation in \eqref{sasts}, a practical way of 
 numerically computing specific instances of expressions of the form $g_{\lambda \, \mu \, \nu}$. 

\begin{example} 
 We may verify using {\tt SageMath} that 
\begin{multline*}
 s_{321} \ast s_{2211} = \\ 
 s_{21111} + 2 s_{2211} + s_{222} + 2 s_{3111} + 3 s_{321} + s_{33} + 2 s_{411} + 2 s_{42} + s_{51}. 
\end{multline*}
 From the above expansion together with \eqref{sasts}, we find that $ g_{ 321 \, 2211 \, 411} = 2$. 
\end{example} 

 Basic properties of the binary operation $\ast$ include the relations 
\begin{equation}\label{flatKron}
 s_{(n)} \ast s_{\lambda} = s_{\lambda} \ \ \ \text{and} \ \ \ s_{(1^n)} \ast s_{\lambda} = s_{\lambda^{\operatorname{t}}}
\end{equation}
 for the transpose $\lambda^{\operatorname{t}}$ of an integer partition $\lambda$, together with the equalities $s_{\lambda} \ast
 s_{\mu} = s_{\mu} \ast s_{\lambda} = s_{\lambda^{\operatorname{t}}} \ast s_{\mu^{\operatorname{t}}} 
 = s_{\mu^{\operatorname{t}}} \ast s_{\lambda^{\operatorname{t}}}$. Basic properties concerning Kronecker coefficients include the 
 $\mathfrak{S}_{3}$-symmetry of the indices of Kronecker coefficients (which may be seen from the expansion in \eqref{gtochar} in terms 
 of irreducible symmetric group characters), with the indices of each of the Kronecker coefficients among 
\begin{equation}\label{gsym}
 g_{\lambda \, \mu \, \nu} = g_{\lambda \, \nu \, \mu} = \cdots = g_{ \nu \, \mu \, \lambda }
\end{equation}
 being invariant under $\mathfrak{S}_{3}$-permutations. Further such properties include the conjugation-invariance relation such that 
 $ g_{\lambda \, \mu \, \nu} = g_{\lambda \, \mu^{\operatorname{t}} \, \nu^{\operatorname{t}}}$, together with the evaluations for 
 Kronecker coefficients associated with the sign and trivial characters, with 
\begin{equation}\label{gtrivialrep}
 g_{(k) \, \lambda \, \mu} = \text{{\bf 1}}_{\lambda = \mu} 
\end{equation} 
 and with 
\begin{equation}\label{gsignrep}
 g_{(1^{k}) \, \lambda \, \mu} = \text{{\bf 1}}_{\lambda = \mu^{\operatorname{t}}}. 
\end{equation} 
 A basic property relating the Hall inner product on $\textsf{Sym}$ and the $\ast$-operation on $\textsf{Sym}$ is such that 
\begin{equation}\label{HallKronecker}
 \langle f \ast g, h \rangle = \langle f, g \ast h \rangle, 
\end{equation} 
 and the relation in \eqref{HallKronecker} can be shown to follow from the self-duality relation in \eqref{Schurdual} together with 
 the $\mathfrak{S}_{3}$-symmetry property of the indices of Kronecker coefficients indicated in \eqref{gsym}. 

\subsection{Rosas's formula for two-rowed and hook shapes}\label{secRosas}
 A key to a number of our main results is given by a remarkable evaluation due to Rosas \cite{Rosas2001} for Kronecker coefficients of 
 the form $g_{t \, h \, \nu}$ for a two-rowed partition $t$ and a hook $h$. This evaluation needs to be adapted, as below, for 
 our purposes. 

 For an arbitrary integer partition $\eta$, define 
$$ \operatorname{tail}(\eta) = \begin{cases} 
 (\eta_3, \ldots, \eta_{\ell(\eta)}), 
 & \text{if $\ell(\eta) \geq 3$;} \\
 (), & \text{otherwise.} 
 \end{cases} $$ Again for an arbitrary partition $\eta$, define 
\begin{align*} 
 u(\eta) & = \text{$\#$ of 2's in $\operatorname{tail}(\eta)$}, \\ 
 v(\eta) & = \text{$\#$ of 1's in $\operatorname{tail}(\eta)$}, \\ 
 \text{{\bf n}}_{3}(\eta) & = \begin{cases} 
 0, & \text{if $\ell(\eta) = 1$;} \\
 \eta_2, & \text{if $\ell(\eta) \geq 2$ and $\eta_1 - \eta_2 \leq v(\eta)$;} \\
 u(\eta) + 2, & \text{if $\ell(\eta) \geq 2$ and $\eta_1 - \eta_2 > v(\eta)$,} 
 \end{cases} \\ 
 \text{{\bf n}}_{4}(\eta) & = \begin{cases} 
 \eta_{1}, & \text{if $\ell(\eta) = 1$;} \\
 \eta_1, & \text{if $\ell(\eta) \geq 2$ and $\eta_1 - \eta_2 \leq v(\eta)$;} \\
 u(\eta) + v(\eta) + 2, & \text{if $\ell(\eta) \geq 2$ and $\eta_1 - \eta_2 > v(\eta)$, } 
 \end{cases} \\ 
 \text{{\bf d}}_{1}(\eta) & = \begin{cases} 
 0, & \text{if $\ell(\eta) = 1$;} \\
 v(\eta), & \text{if $\ell(\eta) \geq 2$ and $\eta_1 - \eta_2 \leq v(\eta)$;} \\
 \eta_1 - \eta_2, & \text{if $\ell(\eta) \geq 2$ and $\eta_1 - \eta_2 > v(\eta)$, } 
 \end{cases} \\ 
 \text{{\bf d}}_{2}(\eta) & = \begin{cases} 
 0, & \text{if $\ell(\eta) = 1$;} \\
 u(\eta), & \text{if $\ell(\eta) \geq 2$ and $\eta_1 - \eta_2 \leq v(\eta)$;} \\
 \eta_2 - 2, & \text{if $\ell(\eta) \geq 2$ and $\eta_1 - \eta_2 > v(\eta)$, and } 
 \end{cases} \\ 
 \text{{\bf e}}^{[a]}_{\eta}(c) 
 & = \begin{cases} 
 c + 1, & \text{if $\ell(\eta) = 1$;} \\
 c + 1, & \text{if $\ell(\eta) \geq 2$ and $\eta_1 - \eta_2 \leq v(\eta)$;} \\
 a - 1, & \text{if $\ell(\eta) \geq 2$ and $\eta_1 - \eta_2 > v(\eta)$.} 
 \end{cases}
\end{align*}
 For integers $n_3$, $n_4$, $d_1$, $d_2$, $e$, and $r$, define
\begin{align*} 
 \Phi(n_3, n_4, d_1, d_2; e, r) 
 = & \text{{\bf 1}}_{n_3 \leq r - d_2 - 1 \leq n_4} \text{{\bf 1}}_{d_1 + 2 d_2 < e < d_1 + 2 d_2 +3} + \\
 & \text{{\bf 1}}_{n_3 \leq r - d_2 \leq n_4} \text{{\bf 1}}_{d_1 + 2 d_2 \leq e \leq d_1 + 2 d_2 +3} + \\ 
 & \text{{\bf 1}}_{n_3 \leq r - d_2 + 1 \leq n_4} \text{{\bf 1}}_{d_1 + 2 d_2 < e < d_1 + 2 d_2 +3} - \\ 
 & \text{{\bf 1}}_{n_3 + d_2 + d_1 = r} \text{{\bf 1}}_{d_1 + 2 d_2 + 1 \leq e \leq d_1 + 2 d_2 +2}.
\end{align*} 
 Also, we define $$ \Psi(c,r,t,N) = \text{{\bf 1}}_{r-1 \leq c+1 \leq N-r} \text{{\bf 1}}_{c+1=t} + \text{{\bf 1}}_{r \leq \frac{c+t+2}{2} \leq N - 
 r} \text{{\bf 1}}_{ |c + 1 -t| \leq 1} $$ for integer arguments. Define a \emph{double hook} as an integer partition in a set of the form 
 $$ \operatorname{DH}(N) := \big\{ \lambda \vdash N : \ell(\lambda) \leq 2 \, \text{or} \, \big( \ell(\lambda) \geq 3 \, \text{and} 
 \, \lambda_{3} \leq 2 \big) \big\}. $$ 

 Let $\eta$ be a partition, and let $a$ and $r$ and $c$ be integers. Define ${\Xi}$ so that $$ {\Xi}_{\eta}^{[a]}(0, c) := \text{{\bf 1}}_{\eta = 
 (a, 1^{c + 1})} $$ and, for $1 \leq r \leq \lfloor \frac{N}{2} \rfloor$, define
\begin{multline*} 
 {\Xi}_{\eta}^{[a]}(r, c) = \\
 \begin{cases} 
 \text{{\bf 1}}_{c=0} \text{{\bf 1}}_{r=1}, & \text{if $ \eta = (N)$;} \\
 \text{{\bf 1}}_{\eta = (N - r, r)^{\operatorname{t}}}, & \text{if $\eta = (1^N)$;} \\ 
 \Psi(c,r,t,N), & \text{if $\eta = (h, 1^{t}) \neq (N)$;} \\
 \null & \text{and $h \geq 2$ and $t \geq 1$;} \\ 
 \Phi\big( \text{{\bf n}}_{3}(\eta), \text{{\bf n}}_{4}(\eta), \text{{\bf d}}_{1}(\eta), 
 \text{{\bf d}}_{2}(\eta); 
 \text{{\bf e}}_{\eta}^{[a]}(c), r \big) & \text{if $\eta \in \operatorname{DH}(N)$ 
 and $\ell(\eta) \geq 2$;} \\ 
 \null & \text{and $\eta_2 \geq 2$;} \\
 0, & \text{otherwise.} 
 \end{cases}
\end{multline*}

 The functions $\mathbf n_3$, $\mathbf n_4$, $\mathbf d_1$, etc.\ given above were constructed by reorganizing the case-by-case 
 conditions due to Rosas \cite[Corollary 3 and Theorem 4]{Rosas2001} into a uniform system of inequalities, so that the 
 following result can be shown to be equivalent to Rosas's original statements on Kronecker coefficients equivalent to 
 the left-hand side of \eqref{Rosasequiv} below. 

\begin{theorem}
 (Rosas, 2001) Let $t = (N-r,r)$ with $1 \leq r \leq \lfloor \frac{N}{2} \rfloor$, let $h = (a, 1^{c+1})$ with $a \geq 1$ and $c \geq 0$, 
 and let $\nu \vdash N$. 
 Then the relation 
\begin{equation}\label{Rosasequiv}
 g_{(N-r,r) \, (a, 1^{c+1}) \, \nu } = {\Xi}_{\nu}^{[a]}(r,c) 
\end{equation}
 holds (cf.\ \cite[Corollary 3 and Theorem 4]{Rosas2001}).
\end{theorem}

\subsection{Blasiak's combinatorial interpretation}\label{secBlasiakinter}
 The following preliminaries on Blasiak's combinatorial interpretation provide a streamlined version of this interpretation specific to our 
 paper, and one may compare this with the work of Liu \cite{Liu2017} on a reformulation of Blasiak's construction. 

 Let 
\begin{equation}\label{alphabet}
 \mathcal{A} = \{ 1, 2, \ldots \} \cup \{ \overline{1}, \overline{2}, \ldots \} 
\end{equation}
 be an alphabet, being consistent with Blasiak's notation \cite{Blasiak20162018}. We refer to the members of \eqref{alphabet} as 
 \emph{colored letters}. A \emph{colored word} refers to a word in $\mathcal{A}^{\ast}$, being consistent with Blasiak's terminology. 
 According to Blasiak, the \emph{natural order} on \eqref{alphabet} is such that 
\begin{equation}\label{linearorder}
 \overline{1} < 1 < \overline{2} < 2 \cdots. 
\end{equation} 
 
 For a colored word $w \in \mathcal{A}^{\ast}$, let $m = m(w)$ be the largest integer $z$ such that either $z$ or $\overline{z}$ appears in 
 $w$. The \emph{(unbarred) content} of $w$ refers to the tuple $(t_1, t_{2}, \ldots, t_{m})$ such that $t_i$ is equal to the number of 
 occurrences of $i$ plus the number of occurrences of $\overline{i}$ among the characters of $w$, for $i \in \{ 1, 2, \ldots, m \}$
 (cf.\ Section 2.1 in Blasiak's paper \cite{Blasiak20162018}). The \emph{total color} 
 $\operatorname{tc}(w)$ of a colored word $w$ refers to the number of barred letters in $w$. 

 For a word $w \in \mathcal{A}^{\ast}$, define $w^{\operatorname{blft}}$ as the word obtained by moving any barred letters in $w$ 
 to the left, while preserving their relative order, and then removing any bars from the resulting word, thus producing a word 
 in $\{ 1, 2, \ldots \}^{\ast}$. As later clarified, the following \emph{original} definition of a colored tableau from Blasiak 
 \cite{Blasiak20162018} is not equivalent to the definition of the same term in Liu's reformulation of Blasiak's 
 combinatorial interpretation \cite{Liu2017}. 

\begin{definition}\label{coloredBlasiak}
 A \emph{(semistandard) colored tableau} \cite{Blasiak20162018} (cf.\ Section 2.2 in Liu's reformulation \cite{Liu2017}) is a partition 
 tableau labeled with members of $\mathcal{A}$ in such a way so that: 

\begin{enumerate}

\item Unbarred letters weakly increase (from west to east) in each row;

\item unbarred letters strictly increase (from north to south) in each column;

\item barred letters strictly increase in each row; and

\item barred letters weakly increase in each column.

\end{enumerate}
\end{definition}

\begin{example}
 The tableau
\begin{equation}\label{excoloredBlasiak}
 \begin{ytableau}
 \overline{1} & 1 & \overline{2} \\
 \overline{1} & \overline{2} & 2 \\
 \overline{2} & 2 & 3 
\end{ytableau}
\end{equation}
 given by Blasiak \cite{Blasiak20162018} satisfies the conditions in Definition \ref{coloredBlasiak}. 
\end{example}

\begin{remark}\label{notcountercolored}
 Observe that the orderings for unbarred and barred characters in Definition \ref{coloredBlasiak} are treated \emph{separately}, despite 
 the use of the expression ``semistandard,'' which one might think should refer to a semistandard labeling/ordering imposed 
 globally/consistently for \emph{both} barred and unbarred characters, with respect to $<$. So, one might wonder whether or not 
 a tableau such as 
\begin{equation}\label{violatingtableau}
 \begin{ytableau} 1 & \overline{1} \end{ytableau}, 
\end{equation}
 which would seem to violate the imposed relation $\overline{1} < 1$ from \eqref{linearorder}, would be permissible for our purposes. 
 The tableau in \eqref{violatingtableau} \emph{is} a (semistandard) colored tableau according to Blasiak's definition reproduced above 
 in Definition \ref{coloredBlasiak}, but, as we later consider, the specific tableau in \eqref{violatingtableau} does not satisfy further 
 conditions required in Blasiak's combinatorial interpretation. In Liu's reformulation \cite{Liu2017} of Blasiak's combinatorial 
 interpretation (see Section 2.2), the same term ``(semistandard) colored tableau'' \emph{is} defined so that the rows and columns 
 \emph{are} globally/consistently weakly increasing with respect to a linear order such as \eqref{linearorder}. 
\end{remark}

 We require the usual row/Schensted insertion procedure associated with the RSK bijection, as reviewed below. The following 
 insertion procedure is the usual \emph{semistandard row insertion} procedure, which one may compare against Sagan's definition
 of insertion for \emph{partial tableaux} with distinct entries \cite[pp.\ 92--93]{Sagan2001}, and referring to the appropriate text by 
 Fulton for further preliminaries and background material concerning the below procedure \cite{Fulton1997}. 

\begin{definition}\label{defineSchensted}
 Let $T$ be an SSYT of partition shape. What we refer to as \emph{ordinary row insertion} or \emph{Schensted insertion} refers to the 
 following insertion procedure. Insert $x$ into the initial row ($i=1$) for a positive integer $x$. Then: 

\begin{enumerate}

\item Determine the leftmost entry/label $y$ in the $i^{\text{th}}$ row such that $y > x$; and 

\item Replace $y$ by $x$, and then insert $y$ into row $i+1$ using the same process. 

\end{enumerate}

\noindent If no such label $y$ exists in a given row, we append $x$ to the end of this row, and the process is terminated. We let the 
 resultant tableau be denoted as $ T\leftarrow x$. 
\end{definition}

\begin{example}
 Set $$ T = \begin{ytableau} 1 & 2 & 4 \\ 2 & 3 \\ 4 \end{ytableau} $$ and let $x = 3$, according to the notation in Definition 
 \ref{defineSchensted}. In the initial row, the leftmost label strictly greater than $x$ is $4$. We replace this greater label with $x$, and 
 bump $4$ into the next lower row, producing $$\begin{ytableau} 1 & 2 & 3 \\ 
 2&3\\ 4 \end{ytableau}, $$ with the understanding that $4$ is to be Schensted-inserted into the second row. Since there is no label in 
 this second row strictly greater than $4$, we position $4$ at the end of this row, yielding 
\begin{equation*}
 T\leftarrow x = \begin{ytableau} 1&2&3\\ 2&3&4\\ 4 \end{ytableau}. 
\end{equation*}
\end{example}

 The analogue of Definition \ref{defineSchensted} referred to as \emph{column insertion} may be defined by taking the transpose of a 
 given tableau, applying Definition \ref{defineSchensted}, 
 and then transposing the resultant tableau. 
 The notion of \emph{mixed insertion} is key to Blasiak's construction 
 \cite{Blasiak20162018} and traces back to the work of Haiman in 1989 \cite{Haiman1989}.

\begin{definition}\label{mixedinsertion} 
 (Mixed insertion) \, Let $w=w_1\cdots w_n$ be a colored word and let $T_0$ be a colored tableau. We construct tableaux $T_0, 
 T_1, \dots, T_n$ recursively, in the following manner. 

 If $w_i$ is unbarred, then we insert this character into the initial row of $T_{i-1}$. If $w_i$ is barred, then we insert it into the initial 
 column of $T_{i-1}$. Whenever a character $\alpha$ is bumped, we insert $\alpha$ into the row immediately underneath if $\alpha$ is 
 unbarred, and into the column immediately to the right if $\alpha$ is barred. We then continue until no new letter is bumped. The 
 resultant tableau $T_n$ is called the \emph{mixed insertion tableau} of $w$ This may be denoted as $T_{0}
 \xleftarrow{\operatorname{m}} w$ \cite[Definition 2.4]{Blasiak20162018} (cf.\ \cite{Haiman1989}). 
\end{definition}

\begin{example}
 Let $T$ denote the colored tableau shown in \eqref{excoloredBlasiak}. Inserting $\overline{1}$ into this tableau, we obtain 
 $$\begin{ytableau} \overline{1} & 1 & \overline{2} \\ \overline{1} & \overline{2} & 2 \\ \overline{1} & \overline{2} & 3 \\ 2 \end{ytableau}.$$

 Inserting $1$ into $T$ according to Definition \ref{mixedinsertion}, we obtain 
$$\begin{ytableau} \overline{1} & 1 & 1 & \overline{2} \\ \overline{1} & \overline{2} & 2 \\ \overline{2} & 2 & 3 \end{ytableau}.$$

 Inserting $\overline{2}$ 
 into $T$ according to Definition \ref{mixedinsertion}, we obtain 
 $$\begin{ytableau} \overline{1} & 1 & \overline{2} \\ \overline{1} & \overline{2} & 2 \\ \overline{2} & 2 & 3 \\ \overline{2} \end{ytableau}.$$

 Inserting $2$ into $T$ according to Definition \ref{mixedinsertion}, we obtain 
 $$ \begin{ytableau} \overline{1} & 1 & \overline{2} & 2 \\ \overline{1} & \overline{2} & 2 \\ \overline{2} & 2 & 3 \end{ytableau}. $$
\end{example}

\begin{definition}
 Let $w = w_{1} w_{2} \cdots w_{n}$ be a colored word. We then form a sequence $P_{0}$, $P_{1}$, $\ldots$, $P_{n}$ of tableaux recursively, 
 so that $P_{0}$ is the empty tableau, and so that 
\begin{equation}\label{displayPk}
 P_{k} = P_{k-1} \xleftarrow{\operatorname{m}} w_k,
\end{equation}
 where the right-hand side of \eqref{displayPk} is used to indicate the mixed insertion of $w_{k}$ into $P_{k-1}$ according to Definition 
 \ref{mixedinsertion}. We then write $P(w) = P_{n}$ and refer to this final tableau as the \emph{mixed insertion tableau} 
 of $w$ \cite[Definition 2.4]{Blasiak20162018}. 
\end{definition}

\begin{example}
 For the colored word $w = \overline{2} \, 1 \, \overline{4} \, 4 \, \overline{4} \, 3 \, \overline{1} \, 3$, we find that 
 $$ P(\overline{2}1\overline{4}4\overline{4}3\overline{1}3)= \begin{ytableau} \overline{1} & \overline{2} & 3 & 3 \\ 1 & \overline{4} \\ 
 \overline{4} & 4 \end{ytableau}. $$ In this case, we find that $w^{\operatorname{blft}} = 24411433$. 
\end{example}

\begin{definition}
 A colored word $y$ is Yamanouchi if, for every suffix $y_{k+1} y_{k+2}\cdots y_n$, its content forms a partition (see Section 3.1 of 
 Blasiak's work \cite{Blasiak20162018}). 
\end{definition}

 Following Blasiak \cite{Blasiak20162018} (comparing the following streamlined definition against the material in Section 3.1 in Blasiak's 
 paper, and against Liu's simplification \cite{Liu2017} of Blasiak's construction), \emph{colored Yamanouchi tableaux} may be defined
 as members in sets of the form 
\begin{equation}\label{CYTd}
 \operatorname{CYT}_{\lambda, d} = \left\{ T = P(w) \, \Bigg| \, \begin{matrix}
 \text{$w$ is a colored word of content $\lambda$;} \\
 \text{$\operatorname{tc}(w) = d$; and} \\
 \text{$w^{\operatorname{blft}}$ is Yamanouchi.}
 \end{matrix} \right\}. 
\end{equation}
 This is equivalent to the formulation used in Blasiak's hook Kronecker rule \cite[Theorem 3.5]{Blasiak20162018}. 

 For a tableau $T$, we write $T_{\operatorname{SW}}$ in place of its most southwest entry. We impose the condition that the 
 southwest entry is unbarred, as in Blasiak's combinatorial interpretation \cite{Blasiak20162018}. This leads us to the remarkable 
 result due to Blasiak (cf.\ \cite[Theorem 3.5]{Blasiak20162018}) whereby
\begin{equation}\label{Blasiakinterpret}
 g_{\lambda \, (n-d, 1^d) \, \nu} = \#\{ T \in \operatorname{CYT}_{\lambda, d} \, | \, \operatorname{shape}(T) = \nu, \, 
 \text{$T_{\operatorname{SW}}$ is unbarred} \}.
\end{equation}
 For the purposes of this paper, we refer to the tableaux in the set on the right of \eqref{Blasiakinterpret} as \emph{Blasiak tableaux}. 
 More specifically, for the parameters/variables presented in \eqref{Blasiakinterpret}, we refer to the tableaux in the set on the right 
 of \eqref{Blasiakinterpret} as \emph{Blasiak tableaux parameterized by $\lambda$, $(n-d, 1^{d})$, and $\nu$}. 

\begin{example}\label{exglobal}
 After inputting
{\footnotesize
\begin{verbatim}
Sym = SymmetricFunctions(QQ);
s = Sym.schur();
print((s[5,2,1].kronecker_product(s[4,1,1,1,1])).scalar(s[4,2,1,1]));
\end{verbatim}}
\noindent into {\tt SageMath}, we find that 
\begin{equation}\label{geq5}
 g_{521 \, 41111 \, 4211} = 5.
\end{equation} 
 The five Blasiak tableaux corresponding to the evaluation in \eqref{geq5} are as below. $$T_1= \begin{ytableau} \bar1 & 1 & 1 & 1\\ 
 \bar1 & \bar3\\ \bar2\\ 2 \end{ytableau} \ \ \ \ \ \ 
 T_2= \begin{ytableau} \bar1 & 1 & 1 & \bar2\\ \bar1 & 2\\ \bar1\\ 3 \end{ytableau}$$

$$ T_3 = \begin{ytableau} \bar1 & 1 & 1 & \bar2\\ \bar1 & \bar3\\ 1\\ 2 \end{ytableau} \ \ \ \ \ \ 
 T_4 = \begin{ytableau} \bar1 & 1 & 1 & \bar3\\ \bar1 & \bar2\\ 1\\ 2 \end{ytableau}$$

$$ T_5= \begin{ytableau} \bar1 & 1 & 1 & \bar3\\ \bar1 & 2\\ \bar1\\ 2 \end{ytableau} $$

\noindent Choices of colored words corresponding to the above tableaux according to \eqref{CYTd} include 
\begin{align}
 w^{(1)} & = \overline{1} \, 1 \, 1 \, 2 \, 1 \, \overline{1} \, \overline{3} \, \overline{2}, \nonumber \\ 
 w^{(2)} & = \overline{1} \, \overline{2} \, \overline{1} \, 1 \, 3 \, 2 \, 1 \, \overline{1}, \nonumber \\ 
 w^{(3)} & = \overline{2} \, 1 \, 1 \, 2 \, 1 \, \overline{1} \, \overline{3} \, \overline{1}, \label{wAppendix} \\ 
 w^{(4)} & = \overline{3} \, 2 \, 1 \, 1 \, 1 \, \overline{1} \, \overline{2} \, \overline{1}, 
 \ \text{and} \nonumber \\ 
 w^{(5)} & = 2 \, \overline{1} \, \overline{1} \, \overline{3} \, \overline{1} \, 1 \, 2 \, 1. 
\end{align}
 For example, the successive insertion tableaux corresponding to the colored word $w^{(3)}$ in \eqref{wAppendix} include 
\begin{equation}\label{suggestinduct}
 \bar2 
 \rightsquigarrow \begin{ytableau}\bar2\end{ytableau} \ \ \ \bar2\,1 \rightsquigarrow \begin{ytableau}1 & \bar2\end{ytableau}, \ \ 
 \ \text{and} \ \ \ \bar2\,1\,1 \rightsquigarrow \begin{ytableau}1 & 1 & \bar2\end{ytableau}, 
\end{equation}
 and further such illustrations are given in the Appendix, as in Section \ref{AppendInsert} below. 
\end{example}

\begin{remark} 
 Although Blasiak’s original definition imposes semistandard conditions \emph{separately} on barred and unbarred entries, Blasiak 
 tableaux (as defined above) are, in fact, weakly increasing along the rows and columns, i.e., globally with respect to both barred and 
 unbarred characters and with respect to the total order in \eqref{linearorder} (see Example \ref{exglobal}). This can be shown inductively 
 (as may be suggested in \eqref{suggestinduct}) via the repeated applications of mixed insertion (and with the use of the 
 Yamanouchi condition). So, the tableau in \eqref{violatingtableau} (within Remark \ref{notcountercolored}) is not a Blasiak tableau (as 
 may be verified), despite it being a so-called ``semistandard'' colored tableau (according to Blasiak's definition, but not Liu's definition), 
 and this, in turn, is despite it having a row that is not weakly increasing with respect to $\leq$.
\end{remark}

\section{Near-hooks}\label{secNear}
 A \emph{near-hook}, for the purposes of this paper, refers to an integer partition of the form $(a, b, 1^{c})$ for $a \geq b \geq 2$. The 
 following Theorem may be seen as providing the basis for our subsequent results. 
 
\begin{theorem}\label{fundamentaltheorem}
 Let $\nu \vdash n$ and $\lambda \vdash n$, and let $a$, $b$, and $c$ be such that $a \geq b \geq 2$ and $a + b + c = n$. Then 
 $$ g_{\lambda \, (a, b, 1^{c}) \, \nu} = \sum_{\substack{ \eta \vdash n - (b-1) \\ \delta \vdash b-1 \\ \theta \vdash n -(b-1) }} c_{\eta \, 
 \delta}^{\nu} g_{\theta \, (a, 1^{c+1}) \, \eta } c^{\lambda}_{\theta \, \delta} - \sum_{\substack{ \eta \vdash a \\ \delta \vdash n-a \\ 
 \theta \vdash n-a }} c_{\eta \, \delta}^{\nu} g_{\theta \, (b-1, 1^{c+1}) \, \delta } c^{\lambda}_{\eta \, \theta}. $$
\end{theorem}

\begin{proof}
 From the consequence of the Giambelli--Leibniz expansion shown in \eqref{GLgives}, we write 
\begin{equation}\label{rightast}
 s_{(a, b, 1^c)} \ast s_{\nu} = \big( s_{(a, 1^{c+1})} s_{(b-1)} \big) \ast s_{\nu} - \big( s_{(a)} s_{(b-1, 1^{c+1})} \big) \ast s_{\nu}. 
\end{equation}
 In the formulation of the Littlewood identity in \eqref{mainLittlewood}, we set $\lambda = (a, 1^{c+1})$ and $\mu = (b-1)$, yielding 
\begin{equation}\label{setLittlewood}
 \left( s_{(a, 1^{c+1})} s_{(b-1)} \right) \ast s_{\nu} = \sum_{\substack{ \eta \vdash n - (b-1) \\ \delta \vdash b - 1 }} c_{\eta \, \delta}^{\nu} 
 \left( s_{(a, 1^{c+1})} \ast s_{\eta} \right) \left( s_{(b-1)} \ast s_{\delta} \right), 
\end{equation}
 and we find that \eqref{setLittlewood} simplifies to 
\begin{equation}\label{Littlewoodsimplify1}
 \left( s_{(a, 1^{c+1})} s_{(b-1)} \right) \ast s_{\nu} = \sum_{\substack{ \eta \vdash n - (b-1) \\ \delta \vdash b - 1 }} c_{\eta \, \delta}^{\nu} 
 \left( s_{(a, 1^{c+1})} \ast s_{\eta} \right) s_{\delta} 
\end{equation}
 according to \eqref{flatKron}. Now, in Littlewood's identity, we set $\lambda = (a)$ and $\mu = (b-1, 1^{c+1})$, producing the relation 
\begin{equation}\label{Littlewoodfund2}
 \left( s_{(a)} s_{(b-1, 1^{c+1})} \right) \ast s_{\nu} = \sum_{\substack{ \eta \vdash a \\ \delta \vdash n - a }} c_{\eta \, \delta}^{\nu} \left( 
 s_{(a)} \ast s_{\eta} \right) \left( s_{(b-1, 1^{c+1})} \ast s_{\delta} \right), 
\end{equation}
 and we see that \eqref{Littlewoodfund2} simplifies to 
\begin{equation}\label{Littlewoodsimplify2}
 \left( s_{(a)} s_{(b-1, 1^{c+1})} \right) \ast s_{\nu} = \sum_{\substack{ \eta \vdash a \\ \delta \vdash n - a }} c_{\eta \, \delta}^{\nu}
 s_{\eta} \left( s_{(b-1, 1^{c+1})} \ast s_{\delta} \right), 
\end{equation}
 again according to the rule for Kronecker products associated with trivial characters. From the relation in \eqref{sasts}, the coefficient of 
 $s_{\lambda}$ on the left-hand side of \eqref{rightast} is $g_{ \lambda \, (a, b, 1^c) \, \nu }$. Applying $\langle s_{\lambda}, \cdot 
 \rangle$ to both sides of each of \eqref{Littlewoodsimplify1} and \eqref{Littlewoodsimplify2}, the desired result then follows from the 
 Littlewood--Richardson rule together with the expansion rule for Kronecker products of Schur functions in terms of the Schur basis. 
 \end{proof}

\begin{example}
 In Theorem \ref{fundamentaltheorem}, we set $n = 6$, $a = 4$, $b = 2$, and $c = 0$, and we set $\lambda = \nu = 42$. This yields
 the expansion of $g_{42 \, 42 \, 42} = 2$ as 

 \ 

\noindent $ c_{41 \, 1}^{42} g_{41 \, 41 \, 41} c_{41 \, 1}^{42} + c_{41 \, 1}^{42} g_{32 \, 41 \, 41} c_{32 \, 1}^{42} + c_{32 \, 1}^{42}
 g_{41 \, 41 \, 32} c_{41 \, 1}^{42} + c_{32 \, 1}^{42} g_{32 \, 41 \, 32} c_{32 \, 1}^{42} - $

\noindent $ c_{4 \, 2}^{42} g_{2 \, 11 \, 2} c_{4 \, 2}^{42} - c_{31 \, 2}^{42} g_{2 \, 11 \, 2} c_{31 \, 2}^{42} - 
 c_{31 \, 2}^{42} g_{11 \, 11 \, 2} c_{31 \, 11}^{42} - c_{31 \, 11}^{42} g_{2 \, 11 \, 11} c_{31 \, 2}^{42} - $ 
 
\noindent $ c_{31 \, 11}^{42} g_{11 \, 11 \, 11} c_{31 \, 11}^{42} - c_{22 \, 2}^{42} g_{2 \, 11 \, 2} c_{22 \, 2}^{42}$. 

 \ 

\noindent Numerically evaluating the above LR/Kronecker coefficients, in order, we obtain 

 \ 

\noindent $ g_{42 \, 42 \, 42} = 1 \cdot 1 \cdot 1 + 1 \cdot 1 \cdot 1 + 1 \cdot 1 \cdot 1 + 1 \cdot 1 \cdot 1 - 1 \cdot 0 \cdot 1 - 
 1 \cdot 0 \cdot 1 - 
 1 \cdot 1 \cdot 1 - 1 \cdot 1 \cdot 1 - 1 \cdot 0 \cdot 1 - 1 \cdot 0 \cdot 1$. 
\end{example}

 In contrast to the trivial case in \eqref{gtrivialrep} given by a Kronecker coefficient in which at least one of the indexing partitions has only 
 one part, for the case whereby there is even at least one index that consists of two parts, this case is \emph{far} more difficult, to the 
 extent that this can even be seen as the starting point for the whole Kronecker problem. This motivates our investigations below. 

 Informally, the key to ``collapsing'' or reducing sums as in Theorem \ref{fundamentaltheorem} reduces to using the Pieri rule (as 
 formulated in terms of LR tableaux) in conjunction with the sign representation identity for Kronecker coefficients in \eqref{gsignrep}. 

\subsection{Two-rowed shapes paired with near-hooks}\label{twonearhook}
 Global hardness results on the computation of Kronecker coefficients \cite{BurgisserIkenmeyer2008,IkenmeyerMulmuleyWalter2017} 
 suggest that it may not be plausible that there is any meaningful, uniform, LR-type rule for Kronecker coefficients in full generality, and 
 hence the importance of identifying well behaved and well structured subclasses of partition shapes where exact evaluations and 
 combinatorial interpretations are permissible. This motivates our new results and techniques concerning the case of Theorem 
 \ref{fundamentaltheorem} given by $$ g_{(d, e) \, (a, b, 1^{c}) \, \nu} = \sum_{\substack{ \eta \vdash n - (b-1) \\ \delta \vdash b - 
 1 \\ \theta \vdash n -(b-1) }} c_{\eta \, \delta}^{\nu} g_{\theta \, (a, 1^{c+1}) \, \eta } c^{(d, e)}_{\theta \, \delta} - \sum_{\substack{ \eta 
 \vdash a \\ \delta \vdash n-a \\ \theta \vdash n-a }} c_{\eta \, \delta}^{\nu} g_{\theta \, (b-1, 1^{c+1}), \delta } c^{(d, e)}_{\eta \, 
 \theta}, $$ in view of how this relates to the cases and references listed in the survey in Section \ref{secSurvey}. 

 A key to our simplification of the above sums is given by how the LR coefficients $c^{(d, e)}_{\theta \, \delta}$ and $c^{(d, e)}_{\eta \, 
 \theta}$ vanish apart from the cases whereby $\theta$, $\delta$, $\eta$, and $\theta$ are subpartitions of $(d, e)$, and this leads us 
 toward Lemma \ref{LemLRtwo} below. 
 
\begin{lemma}\label{LemLRtwo}
 Let $x \geq y$, $u \geq v$, and $d \geq e$, and let $d+e = x + y + u + v$. Then 
\begin{equation}\label{displayLRtwo}
 c^{(d, e)}_{(x, y) \, (u, v)} = \text{{\bf 1}}_{\max(x+v, y+u) \leq d \leq x + u}. 
\end{equation}
 Equivalently, we have that $ c^{(d, e)}_{(x, y) \, (u, v)} \leq 1$ with equality if and only if there exists an integer $k$ satisfying both $$ 0 
 \leq k \leq \min(x - y, u - v) $$ and $$ (d, e) = (x + u - k, y + v + k). $$
\end{lemma}

\begin{proof} 
 By the Jacobi--Trudi rule, we have that $$ s_{(x, y)} = h_{x} s_{(y)} - h_{x+1} s_{(y-1)}, $$ so that 
\begin{equation}\label{stwostwo}
 s_{(x, y)} s_{(u, v)} = h_{x} \big( s_{(y)} s_{(u, v)} \big) - h_{x+1} \big( s_{(y-1)} s_{(u, v)} \big). 
\end{equation}
 The Pieri rule then gives us that 
\begin{equation}\label{Pieristwo}
 h_{m} s_{(u, v)} = \sum_{t=0}^{\min(m, u-v)} s_{(u+m-t, v+t)} + \sum_{\ell(\mu) \geq 3} a_{\mu} s_{\mu} 
\end{equation}
 for certain coefficients $a_{\mu}$. By applying \eqref{Pieristwo} to the right-hand terms in \eqref{stwostwo}, and by then disregarding any 
 terms associated with $h_x s_\mu$ or $h_{x+1} s_{\mu}$ for $\ell(\mu) \geq 3$ (since we are interested in an LR coefficient with a 
 two-row outer shape), we obtain the linear combination 
\begin{equation}\label{linearstwo}
 \sum_{t=0}^{\min(y, u - v)} h_{x} s_{(u+y-t, v+t)} - \sum_{t=0}^{\min(y-1, u-v)} h_{x+1} s_{(u+y-1-t, v+t)}. 
\end{equation}
 Applying the Pieri rule within both of the summands in \eqref{linearstwo}, and then disregarding any terms indexed by partitions of 
 length at least $3$, a telescoping argument gives us that the remaining terms are of the forms $$ s_{(x+u, y+v)}, \, s_{(x + u - 1, y + v + 
 1)}, \, \ldots, \, s_{(x + u - m, y + v + m)}, $$ writing $m = \min(x - y, u - v)$, so that 
\begin{equation}\label{LRconclude}
 s_{(x, y)} s_{(u, v)} = \sum_{k=0}^{\min(x-y, u-v)} s_{(x + u - k, y + v + k)} + \sum_{\ell(\mu) \geq 3} b_{\mu} s_{\mu}, 
\end{equation}
 for certain coefficients $b_{\mu}$. The Littlewood--Richardson rule applied to \eqref{LRconclude} then gives us an equivalent version
 of the desired result. 
\end{proof}

\begin{example}
 Letting $x = 5$, $y = 2$, $u = 3$, $v = 1$, $d = 6$, and $e = 5$, we find that $c_{52 \, 31}^{65} = 1$, and this corresponds to the unique 
 LR tableau 
\begin{equation}\label{tab65} 
 \begin{ytableau}
 *(green) \null & *(green) \null & *(green) \null & *(green) \null & *(green) \null & *(white) 1 \\ 
 *(green) \null & *(green) \null & *(white) 1 & *(white) 1 & *(white) 2 
\end{ytableau} 
\end{equation}
 where the colored cells in \eqref{tab65} are meant to illustrate the skew shape underlying the given LR tableau. The numerical values 
 associated with the purported inequalities in $ \max(x+v, y+u) \leq d \leq x + u$, are, in this case, such that $6 \leq 6 \leq 8$, 
 illustrating the equivalence in \eqref{displayLRtwo}. 
\end{example}

 We thus have, from Lemma \ref{LemLRtwo}, that LR coefficients indexed by two-rowed partitions are multiplicity-free and are, moreover, 
 determined by a single interval condition on the initial part. 

 We henceforth adopt the convention whereby a partition $\lambda$ concatenated with a tuple of the form $(0, 0, \ldots, 0)$ is identified 
 with $\lambda$. Applying Lemma \ref{LemLRtwo} to Theorem \ref{fundamentaltheorem} in the case of a two-row partition motivates the 
 following compressed notation, writing 
\begin{multline*} 
 \operatorname{triple}^{(1)}_{(d, e) \, (a, b, 1^c) \, \nu} = \sum_{ \substack{ \eta \vdash n - b + 1 \\
 0 \leq j \leq \left\lfloor \frac{b-1}{2} \right\rfloor \\ 
 0 \leq r \leq \left\lfloor \frac{n-b+1}{2} \right\rfloor } } 
 \text{{\bf 1}}_{\max( n - b + 1 - r + j, r + b - 1 - j ) \leq d \leq n - r - j } \times \\
 c_{\eta \, (b-1-j, j)}^{\nu} g_{ (n - b + 1 - r, r) \, (a, 1^{c+1}) \, \eta }
\end{multline*}
 and 
\begin{multline*} 
 \operatorname{triple}^{(2)}_{(d, e) \, (a, b, 1^c) \, \nu} = \sum_{ \substack{ \delta \vdash n - a \\
 0 \leq i \leq \left\lfloor \frac{a}{2} \right\rfloor \\ 
 0 \leq r \leq \left\lfloor \frac{n-a}{2} \right\rfloor } } 
 \text{{\bf 1}}_{\max( a - i + r, i + n - a - r ) 
 \leq d \leq n - i - r } \times \\
 c_{(a-i, i) \, \delta}^{\nu} g_{ (n - a - r, r) \, (b-1, 1^{c+1}) \, \delta}. 
\end{multline*}

\begin{lemma}\label{triple1and2}
 The relation $$ g_{(d, e) \, (a, b, 1^c) \, \nu} = \operatorname{triple}^{(1)}_{(d, e) \, (a, b, 1^c) \, \nu} - \operatorname{triple}^{(2)}_{(d, 
 e) \, (a, b, 1^c) \, \nu} $$ holds for $d \geq e$ and for $a \geq b \geq 2$ and for $c \geq 1$, letting $a + b + c = d + e =n$ and 
 letting $\nu \vdash n$. 
\end{lemma}

\begin{proof}
 This follows in a direct way by setting $\lambda = (d, e)$ in Theorem \ref{fundamentaltheorem} and by then using Lemma \ref{LemLRtwo} 
 to express the resultant LR coefficients of the forms $c_{\theta \, \delta}^{(d, e)}$ and $c_{\eta \, \theta}^{(d, e)}$. 
\end{proof}

\begin{example}
 Set $n =6$, $a = 3$, $b = 2$, $c = 1$, and let $(d, e) = (4, 2)$ and $\nu = (4, 2)$. In this case, we find that 
 $$ \operatorname{triple}^{(1)}_{42 \, 321 \, 42} = \sum_{\eta \vdash 5} c_{\eta \, 1}^{42} g_{41 \, 311 \, \eta} + \sum_{\eta \vdash 
 5} c_{\eta \, 1}^{42} g_{32 \, 311 \, \eta} = 4, $$ and this can be verified with {\tt SageMath}. Similarly, we have that 
\begin{multline*}
 \operatorname{triple}^{(2)}_{42 \, 321 \, 42} 
 = \sum_{\delta \vdash 3} c_{3 \, \delta}^{42} g_{3 \, 111 \, \delta} + \sum_{\delta \vdash 3} c_{3 \, \delta}^{42} 
 g_{21 \, 111 \, \delta} + \\ 
 \sum_{\delta \vdash 3} c_{21 \, \delta}^{42} g_{3 \, 111 \, \delta} + 
 \sum_{\delta \vdash 3} c_{21 \, \delta}^{42} g_{21 \, 111 \, \delta} = 2. 
\end{multline*}
 The above evaluations agree with the evaluation $g_{42 \, 321 \, 42} = 2$.
\end{example}

 Our goal, at this point, is to determine a complete characterization of triples $(\eta, j, r)$ associated with the above triple sum expansion 
 for the expression $\operatorname{triple}^{(1)}_{(d, e) \, (a, b, 1^c) \, \nu}$ such that the associated summand $$ \text{{\bf 1}}_{\max( n 
 - b + 1 - r + j, r + b - 1 - j ) \leq d \leq n - r - j } \, c_{\eta \, (b-1-j, j)}^{\nu} g_{ (n - b + 1 - r, r) \, (a, 1^{c+1}) \, \eta }$$ is nonvanishing 
 (i.e., is strictly positive). 
 
 Let $\eta \subseteq \nu$. Let $|\kappa / \eta| = s^{(1)}$. Let $|\nu / \kappa| = s^{(2)}$. We then define 
 \begin{multline*}
 \mathcal{N}_{\nu \, \eta \, s^{(1)} \, s^{(2)}} = \# \{ \kappa : \eta \subseteq \kappa \subseteq \nu, \kappa/\eta \, 
 \text{horizontal strip of size} \, s^{(1)}, \\ 
 \ \nu/\kappa \, \text{horizontal strip of size} \, s^{(2)} \}. 
\end{multline*}

\begin{lemma}\label{LrtoNN}
 For $\eta \vdash N =n-b+1$ and for $j$ such that $0 \leq j \leq \lfloor \frac{b-1}{2} \rfloor$, the relation $$ c_{\eta \, (b-1-j,j)}^{\nu} = 
 \mathcal{N}_{\nu \, \eta \, j \, b - 1 - j} - \mathcal{N}_{\nu \, \eta \, j - 1 \, b - j} $$ holds, with the understanding that $\mathcal{N}_{\nu 
 \, \eta \, -1 \, b} = 0$. 
\end{lemma}

\begin{proof} 
 By the Jacobi--Trudi rule, we find that 
\begin{equation}\label{byJT}
 s_{(b-1-j, j)} = h_{b-1-j} h_{j} - h_{b-j} h_{j-1}, 
\end{equation}
 and \eqref{byJT} then gives us that $$ s_{\eta} s_{(b-1-j,j) } = s_{\eta} h_{j} h_{b - 1 - j} - s_{\eta} h_{j-1} h_{b - j}. $$ The Pieri rule then 
 gives us that the coefficient of $s_{\nu}$ in $s_{\eta} h_{j} h_{b - 1 - j}$ is the number of partitions $\kappa$ such that: 

\begin{enumerate}

\item The partition $\kappa$ differs from $\eta$ by a horizontal strip of size $j$; and

\item The partition $\nu$ differs from $\kappa$ by a horizontal strip of size $b - 1 - j$. 

\end{enumerate}

\noindent That is, the number of partitions $\kappa$ of the desired form is $\mathcal{N}_{\nu \, \eta \, j \, b-1-j}$. In a similar fashion 
 the coefficient of $s_{\nu}$ in $s_{\eta} h_{j-1} h_{b - j} $ is $\mathcal{N}_{\nu \, \eta \, j-1 \, b-j}$, and hence the desired result. 
\end{proof}

\begin{example}\label{exnotation}
 Let $n = 11$, $b = 5$, $j = 1$, $\eta = 43$, and $\nu = 5321$, noting that $N = 7$. In this case, the evaluation $c_{\eta \, (b - 1 - j, 
 j)}^{\nu} = c_{43 \, 31}^{5321} = 1$ corresponds to the unique LR tableau 
\begin{equation}\label{miscuniqueLR}
 \begin{ytableau}
 *(green) \null & *(green) \null & *(green) \null & *(green) \null & *(white) 1 \\ 
 *(green) \null & *(green) \null & *(green) \null \\ 
 *(white) 1 & *(white) 1 \\ 
 *(white) 2 
\end{ytableau}. 
\end{equation}
 In this case, the expression $ \mathcal{N}_{5321 \, 43 \, 1 \, 3} $ is associated with combinatorial objects
 associated with the same skew shape 
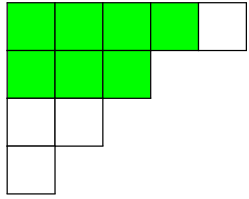
\begin{equation*}
 \begin{ytableau}
 *(green) \null & *(green) \null & *(green) \null & *(green) \null & *(white) \null \\ 
 *(green) \null & *(green) \null & *(green) \null \\ 
 *(white) \null & *(white) \null \\ 
 *(white) \null 
\end{ytableau} 
\end{equation*}
 corresponding to the LR tableau displayed in \eqref{miscuniqueLR}. To evaluate $ \mathcal{N}_{5321 \, 43 \, 1 \, 3}$, we seek partitions 
 $\kappa$ such that $\eta \subseteq \kappa \subseteq \nu$ and such that $\kappa/\eta$ is a horizontal strip of size $1$ and 
 $\nu/\kappa$ is a horizontal strip of size $3$. The unique such partition is $\kappa = 431$ and is illustrated via the colored cells in 
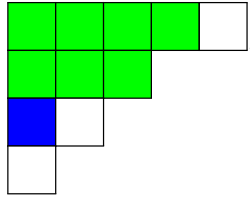
\begin{equation*}
 \begin{ytableau}
 *(green) \null & *(green) \null & *(green) \null & *(green) \null & *(white) \null \\ 
 *(green) \null & *(green) \null & *(green) \null \\ 
 *(blue) \null & *(white) \null \\ 
 *(white) \null 
\end{ytableau}, 
\end{equation*}
 and hence the evaluation $ \mathcal{N}_{5321 \, 43 \, 1 \, 3} = 1$, and we may similarly illustrate that $ \mathcal{N}_{5321 \, 43 \, 0 \, 
 4} = 0$. 
\end{example}

\begin{lemma}\label{gposLem}
 Let $\eta \vdash N = n - b + 1$, and let $a \geq 1$, $b \geq 1$, $c \geq 0$, $0 \leq r \leq \lfloor \frac{N}{2} \rfloor$, 
 and $n = a + b + c$. Then 
\begin{equation*}
 g_{(N-r, r) \, (a, 1^{c+1}), \eta} > 0 \Longleftrightarrow 
 \big( \eta \in \operatorname{DH}(N) \ \text{and} \ {\Xi}^{[a]}_{\eta}(r,c) > 0 \big). 
\end{equation*}
\end{lemma}

 \begin{proof}
 This follows directly from Rosas's formula, as stated in Section \ref{secRosas}, together with the definition of $\Xi^{[a]}_\eta(r,c)$. 
\end{proof}

\begin{example}
 Let $a = 2$, $b = 2$, $c = 5$, $n=9$, and $r = 2$, with $\eta = 32111$ and $N = 8$. Then $$ g_{(N-r,r) \, (a, 1^{c+1} ) \, \eta } = g_{62 
 \, 2111111 \, 32111} = 1 > 0, $$ noting the double hook shape of $\eta$, and we can check that ${\Xi}^{[2]}_{32111}(2,5) > 0$. 
 In this case, writing $$ {\Xi}_{\eta}^{[a]}(r, c) = \Phi\big( \text{{\bf n}}_{3}(\eta), 
 \text{{\bf n}}_{4}(\eta), \text{{\bf d}}_{1}(\eta), \text{{\bf d}}_{2}(\eta); \text{{\bf e}}_{\eta}^{[a]}(c), r \big) = \Phi(2,3,3,0;6,2), $$ a 
 routine calculation shows that this reduces to $ \Phi(2,3,3,0;6,2) = \text{{\bf 1}}_{2 \leq r \leq 3}$.
\end{example}

 Define $N:= n - b + 1$. We also define 
\begin{multline*}
 \mathcal{I}^{+}(\nu) = \Bigg\{ (\eta, j, r) : \eta \in \operatorname{DH}(N), 
 0 \leq j \leq \left\lfloor \frac{b-1}{2} \right\rfloor, 
 0 \leq r \leq \left\lfloor \frac{N}{2} \right\rfloor, \\
 \mathcal{N}_{\nu \, \eta \, j \, b - 1 - j } > \mathcal{N}_{\nu \, \eta \, j - 1 \, b - j}, 
 {\Xi}^{[a]}_{\eta}(r, c) > 0 \Bigg\}.
\end{multline*}

\begin{lemma}\label{iffpos}
 Let $\nu \vdash n$. For $\eta \vdash N := n - b + 1$, for $j$ such that $0 \leq j \leq \lfloor \frac{b-1}{2} \rfloor$, for $r$ such that $0 
 \leq r \leq \lfloor \frac{N}{2} \rfloor$, for $a \geq 1$, $b \geq 1$, and 
 $c \geq 0$, and for $n = a + b + c$, we have that $$ c_{\eta \, (b - 
 1 - j, j)}^{\nu} g_{(N-r,r) \, (a, 1^{c+1}) \, \eta} > 0 \Longleftrightarrow (\eta, j, r) \in \mathcal{I}^{+}(\nu). $$
\end{lemma}

\begin{proof}
 This follows in a direct way from Lemmas \ref{LrtoNN} and \ref{gposLem}. 
 \end{proof}

\begin{example}\label{exanalogy}
 Let $a =3$, $b = 5$, $c = 3$, $j = 1$, $n = 11$, $r = 2$, and set $\nu = 5321$ and $\eta = 421$, with $N = 7$ according to the notation in 
 Lemma \ref{iffpos}. Observe that $\eta$ satisfies the $\operatorname{DH}$-conditions, and that the required bounds on $j$ and $r$ 
 hold. In this case, we are concerned with the value $$ c_{421 \, 31}^{5321} g_{52 \, 31111 \, 421 } = 3 \cdot 1. $$ The value $c_{421 \, 
 31}^{5321}$, according to Lemma \ref{LrtoNN}, may be rewritten so that $c_{421 \, 31}^{5321} = \mathcal{N}_{5321 \, 421 \, 1 \, 
 3} - \mathcal{N}_{5321 \, 421 \, 0 \, 4}$, with $ \mathcal{N}_{5321 \, 421 \, 1 \, 3} = 4$ and $\mathcal{N}_{5321 \, 421 \, 0 \, 4} = 1$. 
 Being consistent with our notation in Example \ref{exnotation}, the combinatorial objects associated with the valuation 
 $ \mathcal{N}_{5321 \, 421 \, 1 \, 3} = 4$ are illustrated below. 

$$ 
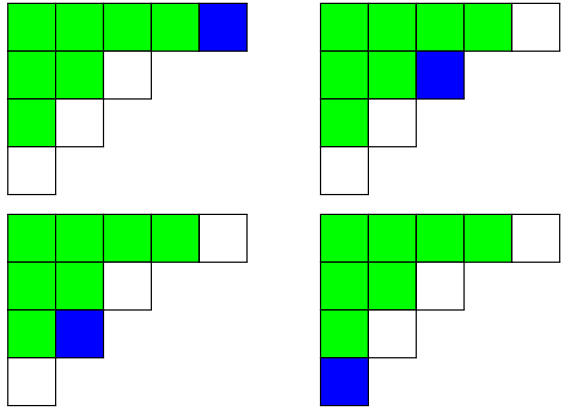
\begin{ytableau}
 *(green) \null & *(green) \null & *(green) \null & *(green) \null & *(blue) \null \\ 
 *(green) \null & *(green) \null & *(white) \null \\ 
 *(green) \null & *(white) \null \\ 
 *(white) \null 
\end{ytableau} 
 \ \ \ \ \ \ \ \begin{ytableau}
 *(green) \null & *(green) \null & *(green) \null & *(green) \null & *(white) \null \\ 
 *(green) \null & *(green) \null & *(blue) \null \\ 
 *(green) \null & *(white) \null \\ 
 *(white) \null 
\end{ytableau} 
 $$

$$ \begin{ytableau}
 *(green) \null & *(green) \null & *(green) \null & *(green) \null & *(white) \null \\ 
 *(green) \null & *(green) \null & *(white) \null \\ 
 *(green) \null & *(blue) \null \\ 
 *(white) \null 
\end{ytableau} 
 \ \ \ \ \ \ \ \begin{ytableau}
 *(green) \null & *(green) \null & *(green) \null & *(green) \null & *(white) \null \\ 
 *(green) \null & *(green) \null & *(white) \null \\ 
 *(green) \null & *(white) \null \\ 
 *(blue) \null 
\end{ytableau} $$

\noindent One may similarly illustrate the valuation $\mathcal{N}_{5321 \, 421 \, 0 \, 4} = 1$, and one may check that ${\Xi}^{[3]}_{421}(2, 
 3) =1 > 0$. 
\end{example}

 Set $M := n - a$. Define 
\begin{multline*}
 \mathcal{I}^{-}(\nu) = \Bigg\{ (\delta, i, r) : \delta \in \operatorname{DH}(M), 
 \, 0 \leq i \leq \left\lfloor \frac{a}{2} \right\rfloor, 
 0 \leq r \leq \left\lfloor \frac{M}{2} \right\rfloor, \\ 
 \mathcal{N}_{\nu \, \delta \, i \, a - i} > 
 \mathcal{N}_{\nu \, \delta \, i - 1 \, a - i + 1}, {\Xi}^{[b-1]}_{\delta}(r,c) > 0 \Bigg\}.
\end{multline*} 

\begin{lemma}\label{iffneg}
 Let $\nu \vdash n$. For $\delta \vdash M = n - a$, for $i$ such that $0 \leq i \leq \lfloor \frac{a}{2} \rfloor$, and for $r$ such that $0 \leq r 
 \leq \lfloor \frac{M}{2} \rfloor$, we have that $$ c_{(a-i, i) \, \delta}^{\nu} g_{(M-r, r) \, (b-1, 1^{c+1}) \, \delta} > 0 
 \Longleftrightarrow (\delta, i, r) \in \mathcal{I}^{-}(\nu). $$
\end{lemma}

\begin{proof}
 This follows in a direct way from Lemmas \ref{LrtoNN} and \ref{gposLem}. 
\end{proof}

\begin{example}
 Let $a = 3$, $b = 3$, $c = 2$, $i = 1$, $n = 8$, and $r = 2$, and set $\nu = 4211$ and $\delta = 221$. In this case, we are concerned with 
 the value $$ c_{21 \, 221}^{4211} g_{32 \, 2111 \, 221} = 1, $$ together with the values of $\mathcal{N}_{4211 \, 221 \, 1 \, 2} = 2$ and
 $\mathcal{N}_{4211 \, 221 \, 0 \, 3} = 1$. The combinatorial objects associated with the valuation $\mathcal{N}_{4211 \, 221 \, 1 \, 
 2} = 2$ may be illustrated with $$ 
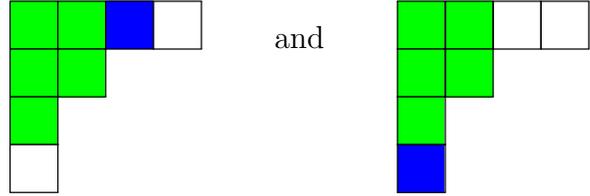
\begin{ytableau}
 *(green) \null & *(green) \null & *(blue) \null & *(white) \null \\ 
 *(green) \null & *(green) \null \\ 
 *(green) \null \\ 
 *(white) \null 
\end{ytableau} \ \ \ \ \ \ \ \text{and} \ \ \ \ \ \ \ \begin{ytableau}
 *(green) \null & *(green) \null & *(white) \null & *(white) \null \\ 
 *(green) \null & *(green) \null \\ 
 *(green) \null \\ 
 *(blue) \null 
\end{ytableau}, $$ by analogy with Example \ref{exanalogy}, and one may similarly illustrate $\mathcal{N}_{4211 \, 221 \, 0 \, 3} = 1$. We 
 may also check that $\Xi^{[b]}_{221}(2,2) = 1 > 0$. 
\end{example}

 Let $n = a + b + c$. As above, define $N:= n - b + 1$ and $M := n - a$. Consider the following refinements of the $\mathcal{I}$-index 
 sets above, writing 
\begin{multline*}
 \mathcal{J}^{+}_{d}(\nu) = \big\{ (\eta, j, r) \in \mathcal{I}^{+}(\nu) : \\ 
 \max(N - r + j, r + b - 1 - j) \leq d \leq N +b - 1 - r - j \big\}
\end{multline*}
 and 
\begin{multline*}
 \mathcal{J}^{-}_{d}(\nu) = \big\{ (\delta, i, r) \in \mathcal{I}^{-}(\nu) : \max(a-i+r, i+M-r) \leq d \leq a + M - i - r \big\}
\end{multline*}
 Now, consider the following refinements of the $\operatorname{triple}^{(1)}$- and $\operatorname{triple}^{(2)}$-sums given 
 above, writing 
\begin{equation}\label{rewritetriple3} 
 \operatorname{triple}^{(3)}_{(d, e) \, (a, b, 1^{c}) \, \nu } = \sum_{ (\eta, j, r) \in \mathcal{J}^{+}_{d}(\nu) } c_{\eta \, (b - 
 1 - j, j)}^{\nu} g_{(N - r, r) \, (a, 1^{c+1}) \, \eta } 
\end{equation}
 and 
\begin{equation}\label{rewritetriple4}
 \operatorname{triple}^{(4)}_{(d, e) \, (a, b, 1^{c}) \, \nu} = \sum_{(\delta, i, r) \in \mathcal{J}^{-}_{d}(\nu) } c_{(a - i, i) \, 
 \delta }^{\nu} g_{(M - r, r) \, (b - 1, 1^{c + 1}) \, \delta}. 
\end{equation}

\begin{remark}
 Observe that the parameter/variable $e$ involved in \eqref{rewritetriple3} and \eqref{rewritetriple4} and in the Theorem below is not 
 \emph{directly} involved in the above definiti-ons/constructions for the index sets 
 $\mathcal{J}^{+}_{d}(\nu)$ and $ \mathcal{J}^{-}_{d}(\nu)$. This is because it is \emph{implicit} that $d+e=n$, and this is made explicit 
 in Theorem \ref{triple3minus4} below. 
\end{remark}

\begin{theorem}\label{triple3minus4}
 The relation $$ g_{(d, e) \, (a, b, 1^c) \, \nu} = \operatorname{triple}^{(3)}_{(d, e) \, (a, b, 1^c) \, \nu} - \operatorname{triple}^{(4)}_{(d, e)
 \, (a, b, 1^c) \, \nu} $$ holds for $d \geq e$ and for $a \geq b \geq 2$ and for $c \geq 1$, letting $a + b + c = d + e =n$ and letting 
 $\nu \vdash n$ (letting it be understood, as above, that $N = n - b + 1$ and $M = n - a$, in the index sets $\mathcal{J}_{d}^{+}(\nu)$ 
 and $\mathcal{J}_{d}^{-}(\nu)$). 
\end{theorem}

\begin{proof}
 This follows from Lemmas \ref{triple1and2}, \ref{iffpos}, and \ref{iffneg}. 
\end{proof}

\begin{example}\label{extriple3and4}
 Letting $a = 3$, $b = 2$, $c = 2$, $d = 4$, $e=3$, and $n =7$, and letting $\nu = 3211$, we find that 
\begin{multline*}
 \mathcal{J}_{d}^{+} = \big\{ (321, 0, 2), (321, 0, 3), (3111, 0, 2), (3111, 0, 3), \\ 
 (2211, 0, 2), (2211, 0, 3) \big\}
\end{multline*}
 and that the associated sum $ \operatorname{triple}^{(3)}_{(d, e) \, (a, b, 1^{c}) \, \nu }$ satisfies 
\begin{multline*}
 \operatorname{triple}_{43 \, 3211 \, 3211}^{(3)} = c_{321 \, 1}^{3211} g_{42 \, 3111 \, 321} + c_{321 \, 1}^{3211} g_{33 \, 3111 \, 
 321} + c_{3111 \, 1}^{3211} g_{42 \, 3111 \, 3111} + \\ 
 c_{3111 \, 1}^{3211} g_{33 \, 3111 \, 3111} + c_{2211 \, 1}^{3211} g_{42 \, 3111 \, 2211} + c_{2211 \, 1}^{3211} g_{33 \, 3111 \, 2211}, 
\end{multline*}
 and this numerically reduces so that 
\begin{equation}\label{Ntriple3}
 \operatorname{triple}_{43 \, 3211 \, 3211}^{(3)} = 1 \cdot 2 + 1 \cdot 1 + 1 \cdot 2 + 1 \cdot 1 + 1 \cdot 1 + 1 \cdot 1, 
\end{equation}
 noting the desired positivity of all of the terms on the right of \eqref{Ntriple3}. Similarly, we find that $$ \mathcal{J}^{-}_{d} = \big\{ (22,
 1, 2), (211, 0, 1), (211, 1, 1) \big\} $$ and that the associated sum $ \operatorname{triple}^{(4)}_{(d, e) \, (a, b, 1^{c}) \, \nu }$ satisfies 
 $$ \operatorname{triple}_{43 \, 3211 \, 3211}^{(4)} = c_{21 \, 22}^{3211} g_{22 \, 1111 \, 22} + c_{3 \, 211}^{3211} g_{31 \, 1111 \, 211} 
 + c_{21 \, 211}^{3211} g_{31 \, 1111 \, 211}, $$ and this numerically reduces so that 
\begin{equation}\label{counterto1}
 \operatorname{triple}_{43 \, 3211 \, 3211}^{(4)} = 1 \cdot 1 + 1 \cdot 1 + 2 \cdot 1, 
\end{equation}
 and this agrees with the desired evaluation $ g_{43 \, 3211 \, 3211} = 4 $ and with Theorem \ref{triple3minus4}. This example shows 
 that, even after restricting to $\mathcal{J}_{d}^{-}(\nu)$, the surviving terms need not all be equal to $1$. 
\end{example}

\section{Combinatorial interpretations}\label{secinter}
 Our strategy, at this point, is to impose restrictions on $g_{(d, e) \, (a, b, 1^c) \, \nu} $ (or, rather, its indices) in order to guarantee that if 
 $(\delta, i, r) \in \mathcal{J}^{-}_{d}(\nu)$, then each term associated with the right-hand side of \eqref{rewritetriple4} is equal to $1$, i.e., 
 so that the whole sum on the right of \eqref{rewritetriple4} is equal to the cardinality of the index set $\mathcal{J}^{-}_{d}(\nu)$, which, 
 as above, has an explicit combinatorial interpretation. This leads us to consider natural restrictions on the cardinality of 
 $\mathcal{J}_{d}^{-}(\nu)$, as below. As a natural place to start, we begin with the case such that $\mathcal{J}_{d}^{-}(\nu)$ is a 
 singleton set, and we later consider the $\mathcal{J}_{d}^{-}(\nu) = \varnothing$ case (which may be seen as a more exceptional case). 

\subsection{Singleton index sets of the form $\mathcal{J}_{d}^{-}(\nu)$}

\begin{lemma}\label{Lemsize1}
 Let $a \geq 2$, $b = 2$, $c \geq 1$, $d$, $e$ and $\mathcal{S}$ be integers and let $\nu$ be a partition of $n$, with $a+b+c = d + e = 
 n$ and $d \geq e$ (and we again let $N = n - b + 1$ and $M = n - a$ be as above). Moreover, let $1 \leq \mathcal{S} \leq \lfloor 
 \frac{c+2}{2} \rfloor$, and suppose that 
\begin{enumerate}

\item $\nu = (a+2, 2^{\mathcal{S}-1}, 1^{c+2-2\mathcal{S}})$; and 

\item $\max(a + \mathcal{S}, c + 2 - \mathcal{S}) \leq d \leq a + c + 2 -\mathcal{S}$.

\end{enumerate}

\noindent Then $\left| \mathcal{J}_{d}^{-}(\nu) \right| = 1$, and, moreover, the unique term (in this case) within
 $ \operatorname{triple}_{(d, e) \, (a, b, 1^c) \, \nu}^{(4)}$ is equal to $1$. 
\end{lemma}

\begin{proof}
 We now isolate a family for which the negative contribution 
\(\operatorname{triple}^{(4)}\) collapses to a single term of value \(1\). Let 
 $$ \delta^{\ast} := \big( 2^{\mathcal{S}}, 1^{c+2-2\mathcal{S}} \big) \ \ \ \text{and} 
 \ \ \ M = n - a = c + 2. $$
 We claim that 
\begin{equation}\label{Jdminussingle}
 \mathcal{J}_{d}^{-}(\nu) = \big\{ (\delta^{\ast}, 0, \mathcal{S} ) \big\} 
\end{equation}
 and that the unique term in $\operatorname{triple}^{(4)}$ corresponding to the singleton index set in \eqref{Jdminussingle} is 
 equal to $1$. 

 Let $(\delta, i, r) \in \mathcal{J}_{d}^{-}(\nu)$. From the containment $\mathcal{J}_{d}^{-}(\nu) \subseteq \mathcal{I}^{-}(\nu)$, we obtain 
 the inequality $$ \mathcal{N}_{\nu \, \delta \, i \, a - i} > \mathcal{N}_{\nu \, \delta \, i - 1 \, a - i + 1}, $$ so that $$ \mathcal{N}_{\nu 
 \, \delta \, i \, a - i} > 0. $$ From the definition of our $\mathcal{N}$-sets, there exists a partition $\kappa$ such that $$ \delta \subseteq 
 \kappa \subseteq \nu, $$ and hence 
\begin{equation}\label{deltacontained}
 \delta \subseteq (a+2, 2^{\mathcal{S}-1}, 1^{c+2-2\mathcal{S}}).
\end{equation}
 Since $|\delta| = M = c + 2$, we can conclude, from \eqref{deltacontained}, that the number of boxes of $\nu$ below the first row is
 $2 (\mathcal{S}-1) + (c + 2 - 2 \mathcal{S}) =c$. So, a given subpartition of $\nu$ of size $c+2$ necessarily consists of $c$ boxes
 under the initial row and $2$ boxes in the initial row. The unique such partition is $$ \delta = (2^\mathcal{S}, 1^{c + 2 - 
 2\mathcal{S}}) = \delta^{\ast}, $$ i.e., so that every element in $\mathcal{J}_{d}^{-}(\nu)$ is of the form $(\delta^{\ast}, i, r)$.
 
 Writing $\delta = \delta^{\ast}$, we see that $ \nu / \delta^{\ast} $ consists of $a$ boxes, and all of these boxes are in the initial row. 
 So, for $i$ such that $0 \leq i \leq a$, there exists a unique intermediate partition $$ \kappa^{(i)} = \big( i + 2, 2^{\mathcal{S}-1}, 
 1^{c + 2 - 2 \mathcal{S}} \big) $$ such that $\kappa^{(i)} / \delta^{\ast}$ is a horizontal strip of size $i$ and $\nu / \kappa^{(i)}$ is a 
 horizontal strip of size $a - i$. We thus obtain that $$ \mathcal{N}_{\nu \, \delta^{\ast} \, i \, a - i} = 1 $$ for $i$ such that $0 \leq i \leq 
 a$. Also, if $i \geq 1$, then $$ \mathcal{N}_{\nu \, \delta^{\ast} \, i - 1 \, a - i + 1 } = 1.$$ From the requirement that $$ \mathcal{N}_{\nu 
 \, \delta^{\ast} \, i \, a - i } > \mathcal{N}_{\nu \, \delta^{\ast} \, i - 1 \, a - i +1}, $$ we find that $ i = 0$. 

 Since $b =2$, we find that $$ g_{(M-r,r) \, (b-1 \, 1^{c+1}) \, \delta^{\ast} } = g_{(M-r, r) \, (1^{M}) \, \delta^{\ast}}. $$ The Kronecker 
 coefficient identity in \eqref{gsignrep} associated with sign representations allows us to conclude that $$ g_{(M-r,r) \, (1^M) \, 
 \delta^{\ast} } = 1 \Longleftrightarrow \delta^{\ast} = (M-r,r)^{\operatorname{t}}. $$ Since $$ (M-\mathcal{S}, 
 \mathcal{S})^t=(2^{\mathcal{S}},1^{M-2\mathcal{S}})=\delta^\ast, $$ we can conclude that the biconditional equivalence 
 $$ \delta^\ast=(M-r,r)^t \Longleftrightarrow r=\mathcal S $$ holds. Consequently, the unique element in 
 the index set $\mathcal{J}_{d}^{-}(\nu)$ is $(\delta^{\ast}, 0, \mathcal{S})$. 

 So, it remains to evaluate the unique term in $\operatorname{triple}^{(4)}_{(d,e)\,(a,b,1^c)\,\nu}$ corresponding to $(\delta^\ast, 0, 
 \mathcal S)$. Since $i=0$, this term is $$ c_{(a,0) \, \delta^\ast}^{\nu} \, g_{(M-\mathcal S,\mathcal S)\,(1^M)\,\delta^\ast}, $$ with 
 the reduction $(a,0)=(a)$. Since $\nu/\delta^\ast$ consists of $a$ boxes, with each such box being in the initial row, the Pieri rule gives 
 us that $ c_{(a)\,\delta^\ast}^{\nu}=1$. Moreover, from the sign representation identity in \eqref{gsignrep}, the equality 
 $$ g_{ (M - \mathcal{S}, \mathcal{S}) \, (1^{M}) \, \delta^{\ast} } = \text{{\bf 1}}_{ (M - \mathcal{S}, \mathcal{S} )^{\operatorname{t}} 
 = \delta^{\ast}} $$ holds true. Since $$ (M - \mathcal{S}, \mathcal{S})^{\operatorname{t}} = (2^{\mathcal{S}}, 1^{M - 2 \mathcal{S}}) = 
 (2^{\mathcal{S}}, 1^{c + 2 - 2\mathcal{S}}) = \delta^\ast, $$ we may obtain that $ g_{(M - \mathcal{S}, \mathcal{S}) \, (1^M) \, 
 \delta^\ast}=1$. As a consequence, the unique term is $1\cdot 1=1$, and hence the equality 
 $ |\mathcal J_d^{-}(\nu)|=1$. So, the unique term in $\operatorname{triple}^{(4)}_{(d,e)\,(a,b,1^c)\,\nu}$
 is equal to $1$, as required. 
\end{proof}

\begin{example}
 Let $a = 2$, $b = 2$, $c = 2$, $d = 4$, $e = 2$, $\mathcal{S} = 2$ (with $\nu = (a+2, 2^{\mathcal{S}-1}, 1^{c+2-2\mathcal{S}})$ as 
 in Lemma \ref{Lemsize1} and with $N = n - b +1$ and $M = n - a$ as before). In this case, 
 the conditions in Lemma \ref{Lemsize1} hold, and we obtain the singleton set
\begin{equation}\label{single2202} 
 \mathcal{J}_{d}^{-} = \big\{ (22, 0, 2) \big\}. 
\end{equation}
 In this case, the unique term corresponding to the singleton index set in \eqref{single2202} is $$ c_{2 \, 22}^{42} g_{22 \, 1111 \, 22} = 
 1, $$ i.e., with $ c_{2 \, 22}^{42} = 1$ and $g_{22 \, 1111 \, 22} = 1$. 
\end{example}

\begin{lemma}\label{LemLR1}
Let $$ \nu = (a+2, 2^{\mathcal{S} - 1},1^{c+2 - 2 \mathcal{S}}). $$ If $(\eta,0,r) \in \mathcal J_d^{+}(\nu)$, 
 then $$ c_{\eta,(1)}^{\nu}=1. $$
\end{lemma}

\begin{proof}
 Let $\nu = (a+2,2^{\mathcal{S}-1}, 1^{c+2-2\mathcal{S}})$, and assume that $$ (\eta, 0, r) \in \mathcal{J}_{d}^{+}(\nu). $$ From the 
 containment $\mathcal{J}_{d}^{+} \subseteq \mathcal{I}^{+}(\nu)$, we may deduce that
\begin{equation}\label{deducecalN}
 \mathcal{N}_{\nu \, \eta \, 0 \, 1} > \mathcal{N}_{\nu \, \eta \, -1 \, 2}, 
\end{equation}
 with the right-hand side of \eqref{deducecalN} vanishing by convention. Consequently, we obtain the positivity of $$ \mathcal{N}_{\nu \, 
 \eta \, 0 \, 1} > 0, $$ and this gives us that $\nu/\eta$ is a single-box horizontal strip, so that the Pieri rule allows us to evaluate 
 $c_{\eta \, (1)}^{(a + 
 2, 2^{\mathcal{S}-1}, 1^{c+2-2\mathcal{S}})} = 1$. 
\end{proof}

\begin{example}
 Setting $a = 2$, $c = 3$, and $\mathcal{S} = 2$, we find that $\nu = 421$. Setting $\eta = 321$, this is obtained by removing a single box 
 from the first row of $\nu$. The Pieri rule then gives us that $c_{\eta \, 1}^{\nu} = 1$, in accordance with Lemma \ref{LemLR1}. 
\end{example}

\begin{lemma}\label{beforeinterpret}
 Let $a \geq 2$, $b = 2$, $c \geq 1$, $d$, $e$ and $\mathcal{S}$ be integers, with $a+b+c = d + e = n$ and $d \geq e$ as above. 
 Moreover, let $1 \leq \mathcal{S} \leq \lfloor \frac{c+2}{2} \rfloor$, and suppose that $\max(a + \mathcal{S}, c + 2 - \mathcal{S}) \leq 
 d \leq a + c + 2 -\mathcal{S}$. Then 
\begin{multline*}
 g_{(d, e) \, (a, 2, 1^c) \, (a+2, 2^{\mathcal{S}-1}, 1^{c+2-2\mathcal{S}})} = \\ 
 - 1 + \sum_{ (\eta, 0, r) \in \mathcal{J}^{+}_{d}( a+2, 2^{\mathcal{S}-1}, 1^{c+2-2\mathcal{S}} ) } 
 g_{(n - 1 - r,r) \, (a, 1^{c+1}) \, \eta }.
\end{multline*}
\end{lemma}

\begin{proof}
 This follows in a direct fashion from Theorem \ref{triple3minus4}, Lemma \ref{Lemsize1}, and Lemma \ref{LemLR1}. 
\end{proof}

\begin{example}
 Let $a = 3$, $b = 2$, $c=2$, $d=5$, $e=2$, $\mathcal{S} = {2}$. The given conditions in Lemma \ref{beforeinterpret} hold. In this case, 
 we have that $$ g_{(d, e) \, (a, 2, 1^c) \, (a+2, 2^{\mathcal{S}-1}, 1^{c+2-2\mathcal{S}})} = g_{52 \, 3211 \, 52} = 0 $$ and that 
 $$ \mathcal{J}^{+}_{d} = \big\{ (42, 0, 2) \big\}, $$ with Lemma \ref{beforeinterpret} giving that $$ g_{52 \, 3211 \, 52} = -1 + g_{42 \, 
 3111 \, 42} = -1 + 1. $$ The unique Blasiak tableau corresponding to the valuation $g_{42 \, 3111 \, 42} = 1$ is 
\begin{equation*}
 T = \begin{ytableau} \overline{1} & 1 & 1 & \overline{2} \\ 1 & \overline{2} \end{ytableau}. 
\end{equation*}
 \end{example}

 So, if we let $\mathcal{B}_{\lambda \, h \, \nu}$ denote the set of Blasiak tableaux associated with $g_{\lambda \, h \, \nu}$, this leads us 
 toward the following result. By construction, 
 the expression $ \left| \mathcal{B}_{ (n - 1 - r,r) \, (a, 1^{c+1}) \, \eta } \right| $ in the below summand
 is guaranteed to be positive, for tuples $(\eta, 0, r)$ in the specified index
 set of the form $\mathcal{J}_{d}^{+}$. 

\begin{theorem}\label{mainresult}
 Let $a \geq 2$, $b = 2$, $c \geq 1$, $d$, $e$ and $\mathcal{S}$ be integers, with $a+b+c = d + e = n$ and $d \geq e$. Moreover, 
 let $1 \leq \mathcal{S} \leq \lfloor \frac{c+2}{2} \rfloor$, and suppose that $\max(a + \mathcal{S}, c + 2 - \mathcal{S}) \leq d \leq a + c 
 + 2 -\mathcal{S}$. Then $$ g_{(d, e) \, (a, 2, 1^c) \, (a+2, 2^{\mathcal{S}-1}, 1^{c+2-2\mathcal{S}})} = \\ 
 - 1 + \sum_{\substack{ (\eta, 0, r) \in \\ \mathcal{J}^{+}_{d}( a+2, 2^{\mathcal{S}-1}, 1^{c+2-2\mathcal{S}} ) } } 
 \left| \mathcal{B}_{ (n - 1 - r,r) \, (a, 1^{c+1}) \, \eta } \right|, $$
 with each set of Blasiak tableaux of the form $ \mathcal{B}_{ (n - 1 - r,r) \, (a, 1^{c+1}) \, \eta }$ being nonempty. 
\end{theorem}

\begin{proof}
 This follows from Lemmas \ref{iffpos}
 and \ref{beforeinterpret} and from Blasiak's combinatorial interpretation. 
\end{proof}

\begin{example}\label{exaftermain}
 Set $a = 6$, $b = 2$, $c = 6$, $d = 8$, $e = 6$, and $\mathcal{S} = 2$. In this case, we have that $$ \mathcal{J}_{d}^{+}(821111) = 
 \big\{ (721111, 0, 5), (721111, 0, 6) \big\}, $$ so that $$ g_{86 \, 62111111 \, 821111} = - 1 + g_{85 \, 61111111 \, 721111} + g_{76 \, 
 61111111 \, 721111}. $$ The unique Blasiak tableau corresponding to the valuation 
 $g_{85 \, 61111111 \, 721111} = 1$ is
 $$ { \begin{ytableau} \bar{1} & 1 & 1 & \bar{2} & 2 & 2 & 2\\ \bar{1} & \bar{2}\\ \bar{1}\\ \bar{1}\\ \bar{1}\\ 1 \end{ytableau}} $$
 and the unique Blasiak tableau corresponding to the valuation 
 $g_{76 \, 61111111 \, 721111} = 1$ is 
 $$ \begin{ytableau} \bar{1} & 1 & 1 & \bar{2} & 2 & 2 & 2\\ \bar{1} & \bar{2}\\ \bar{1}\\ \bar{1}\\ \bar{1}\\ 2 \end{ytableau}. $$
\end{example}

 We may reformulate Theorem \ref{mainresult}, in an equivalent way, by defining the disjoint union 
\begin{equation*}
 \mathcal{U}_{d}^{+}\big( a + 2, 2^{\mathcal{S} - 1}, 1^{c+2-2\mathcal{S}} \big) 
 := \bigsqcup_{ (\eta, 0, r) \in \mathcal{J}_{d}^{+}( a + 2, 2^{\mathcal{S} - 1}, 
 1^{c + 2 - 2 \mathcal{S}} ) } 
 \mathcal{B}_{(n - 1 - r, r) \, (a, 1^{c+1}) \, \eta }. 
\end{equation*}
 From Lemma \ref{iffpos}, we find that every component in the above disjoint union is nonempty.
 Fix a total order on the right-hand side of the above equality, 
 by lexicographically sorting tableaux according to partition shapes
 and then by their row-reading words. Let $T_{\operatorname{min}}$ denote the 
 (lexicographically) minimal element. Theorem \ref{mainresult} then gives us that 
 $$ g_{(d, e) \, (a, 2, 1^{c}) \, (a+2, 2^{\mathcal{S} -1}, 1^{c+2-2\mathcal{S}} ) } 
 = \big| \mathcal{U}_{d}^{+}\big( a + 2, 2^{\mathcal{S} - 1}, 1^{c + 2 - 2 \mathcal{S}} \big) 
 \setminus \{ T_{\operatorname{min}} \} \big|. $$ 
 This gives us an expression of the above Kronecker coefficient 
 as the cardinality of a naturally defined family of Blasiak tableaux. 
 
\subsection{Null index sets of the form $\mathcal{J}_{d}^{-}(\nu)$}

\begin{lemma}\label{nulllemma}
 Let $a \geq 2$, $b=2$, $c\geq 1$, and let $a+b+c=d+e=n$ with $d\geq e$. Let $1\leq \mathcal S\leq \lfloor \frac{c+2}{2} \rfloor$, and 
 set $$ \nu = (a+2,2^{\mathcal S-1},1^{c+2-2\mathcal S}). $$ If $$ d < \max(a+\mathcal S,c+2-\mathcal S)
 \quad\text{or}\quad
d>a+c+2-\mathcal S, $$ then $$ \mathcal J_d^{-}(\nu)=\varnothing. $$
Consequently,  the  vanishing relation 
$$ \operatorname{triple}^{(4)}_{(d,e)\,(a,2,1^c)\,\nu}=0. $$  holds. 
\end{lemma}

\begin{proof}
 For $\nu$ as specified, the same argument as in Lemma \ref{Lemsize1} gives us that: 
 Before enforcing 
 the $d$-interval condition in the definition of $\mathcal J_d^{-}(\nu)$, the argument used in Lemma \ref{Lemsize1} shows 
 that the only possible candidate from the negative side is
$$ (\delta^\ast,0,\mathcal S),
       \  \   \  \text{for}  \   \  \  \delta^\ast=(2^{\mathcal S},1^{c+2-2\mathcal S}). $$
 So, for the above 3-tuple, 
 we require that
 $$ \max(a + \mathcal{S}, c + 2 - \mathcal{S}) \leq d \leq a + c + 2 - \mathcal{S}. $$
 If $d$ is not in the interval suggested above, 
 the unique permissible triple would not be in $\mathcal{J}_{d}^{-}(\nu)$. 
\end{proof}

\begin{example}
 Let $n = 8$, with $a = 3$, $b = 2$, $c = 3$, $d = 4$, $e = 4$, and $\mathcal{S} = 2$, giving us that $\nu = 521$. We may check that the 
 desired conditions in Lemma \ref{nulllemma} hold. In this case,   $\mathcal{J}_{d}^{+}(\nu)$ is a singleton set 
 consisting of $(421, 0, 3)$, and we find that the 
 sum of the form $ \operatorname{triple}^{(3)} $ reduces to 
 $$ c_{421 \, 1}^{521} g_{43 \, 31111 \, 421} = 1 \cdot 1. $$
 The unique Blasiak tableau corresponding to the valuation 
 $ g_{43 \, 31111 \, 421} = 1 $ is 
 $$ \begin{ytableau} \overline{1} & 1 & 1 & \overline{2} \\ \overline{1} & \overline{2} \\ 2 \end{ytableau}. $$
 \end{example}

\begin{lemma}\label{nullsecondlemma}
 Let $a \geq 2$, $b=2$, $c\geq 1$, and let $a+b+c=d+e=n$ with $d\geq e$.
 Suppose that $1\leq \mathcal{S}  \leq \lfloor \frac{c+2}{2} \rfloor$
 and that $ d < \max(a+\mathcal{S}, c + 2 - \mathcal{S})$ or 
$d>a+c+2-\mathcal S$. 
 Then the relation 
 $$ g_{(d, e) \, (a, 2, 1^c) \, (a + 
 2, 2^{\mathcal{S} - 
 1},1^{c+2-2\mathcal{S}})} = 
 \sum_{ (\eta, 0, 
 r) \in \mathcal{J}^{+}_{d}(a + 
 2,2^{\mathcal{S} -   1},1^{c+2-2\mathcal{S}}) } 
 g_{(n - 1 - r, r) \, (a, 1^{c+1}) \, \eta }   $$ holds true. 
\end{lemma}

\begin{proof}
 This follows in a direct way from Theorem \ref{triple3minus4}, Lemma \ref{LemLR1},   and Lemma \ref{nulllemma}. 
\end{proof}

\begin{example}
 Let $n = 9$, with $a = 3$, $b = 2$, $c = 4$, $d = 5$, $e = 4$, and $\mathcal{S} = 3$, giving us that $\nu = 522$. We may verify that the 
 desired conditions in Lemma \ref{nullsecondlemma} hold, so that the main identity in Lemma \ref{nullsecondlemma} reduces to 
\begin{equation}\label{nullreduces}
 g_{54 \, 321111 \, 522} = g_{53 \, 311111 \, 422} = 1, 
\end{equation}
 noting that all of the indices on the left-hand Kronecker coefficient in \eqref{nullreduces} are non-hooks. The unique Blasiak tableau 
 corresponding to the valuation $g_{53 \, 311111 \, 422} = 1$ is $$ \begin{ytableau} \overline{1} & 1 & 1 & \overline{2} \\ 
 \overline{1} & \overline{2} \\ 1 & \overline{2} \end{ytableau}, $$ thus providing, via Lemma \ref{nullsecondlemma}, 
 a combinatorial interpretation for the leftmost Kronecker coefficient in \eqref{nullreduces}. 
\end{example}

 We thus arrive at the following counterpart to Theorem \ref{mainresult}, yielding another family of combinatorial interpretations, 
 which is manifestly positive, in this case. 

\begin{theorem}\label{mainresult2}
 Let $a \geq 2$, $b=2$, $c\geq 1$, and let $a+b+c=d+e=n$ with $d\geq e$. Suppose that $1\leq \mathcal S\leq \lfloor \frac{c + 
 2}{2} \rfloor$ and that $ d < \max(a+\mathcal S,c+2-\mathcal S)$ or $d > a + c + 2 - \mathcal{S}$. Then the relation $$ g_{(d, e) \, 
 (a, 2, 1^c) \, (a + 2, 2^{\mathcal{S} - 1}, 1^{c+2 - 2 \mathcal{S}})} = 
 \sum_{ (\eta, 0, r) \in \mathcal{J}^{+}_{d}(a+2,2^{\mathcal S-1},1^{c+2-2\mathcal S}) } 
 \left| \mathcal{B}_{(n - 1 - r, r) \, (a, 1^{c+1}) \, \eta } \right| 
 $$ holds true, 
 with each set of Blasiak tableaux of the form 
 $ \mathcal{B}_{(n - 1 - r, r) \, (a, 1^{c+1}) \, \eta } $ being nonempty. 
\end{theorem}

\begin{proof}
 This follows from Lemmas \ref{nullsecondlemma} and \ref{iffpos} and from Blasiak's combinatorial interpretation. 
\end{proof}

\begin{example}
 A special case of Theorem \ref{mainresult2} gives us that 
\begin{equation}\label{anothergg1}
 g_{(10, 5) \, 72111111 \, 9222} 
 = g_{(10, 4) \, 71111111 \, 8222 } = 1, 
\end{equation} 
 noting the absence of hook partitions among the indices of the leftmost Kronecker coefficient in \eqref{anothergg1}. The unique 
 Blasiak tableau corresponding to the valuation $g_{(10, 4) \, 71111111 \, 8222 } = 1$ is $$ \begin{ytableau}
 \overline{1} & 1 & 1 & 1 & 1 & 1 & 1 & \overline{2} \\ \overline{1} & \overline{2} \\ \overline{1} & \overline{2} \\ 1 & \overline{2} 
 \end{ytableau}, $$ thus giving us, according to Theorem \ref{mainresult2}, 
 a combinatorial interpretation of the leftmost Kronecker coefficient in \eqref{anothergg1}. 
\end{example}

 It is possible to obtain refinements of Theorem \ref{mainresult2}, e.g., by determining conditions to evaluate the right-hand side of 
 the main identity in Theorem \ref{mainresult2} more explicitly, and it is possible to obtain many further results in the spirit of 
 Theorems \ref{mainresult} and \ref{mainresult2}, but we omit further considerations of this, for the time being, as we have 
 demonstrated how our method can be applied to obtain manifestly positive combinatorial interpretations building on Blasiak's 
 combinatorial interpretation, as above. 

\section{Conclusion}
 Our method reduces the difference of triple sums in Theorem \ref{triple3minus4} by identifying conditions under which the
 negative contribution \(\operatorname{triple}^{(4)}\) collapses to a given constant. 
 More broadly, the reduction framework of Theorem \ref{triple3minus4} suggests 
 further applications, for example through sign-reversing involutions or analogous analyses for other families of partition shapes.

\subsection*{Acknowledgements}
 The author thanks Mike Zabrocki for some useful discussions concerning Kronecker coefficients. The author used the GPT-5.3 model to 
 assist with the exposition and with exploratory calculations and to explore heuristic arguments, but 
 all mathematical proofs, results, exposition, etc.\ were/was written by the author and verified independently by the author.

\section{Appendix} 

\subsection{Insertion tableaux for $w^{(3)}$}\label{AppendInsert}
 The successive insertion tableaux corresponding to the colored word $w^{(3)}$ in \eqref{wAppendix} are given in full below. 
\begin{align*}
 \bar2 & \rightsquigarrow \begin{ytableau}\bar2\end{ytableau} \\ 
 \bar2\,1 & \rightsquigarrow \begin{ytableau}1 & \bar2\end{ytableau} \\ 
 \bar2\,1\,1 & \rightsquigarrow \begin{ytableau}1 & 1 & \bar2\end{ytableau} \\ 
 \bar2\,1\,1\,2 & \rightsquigarrow \begin{ytableau}1 & 1 & \bar2 & 2\end{ytableau} \\ 
 \bar2\,1\,1\,2\,1 & \rightsquigarrow \begin{ytableau}1 & 1 & 1 & \bar2\\ 2\end{ytableau} \\ 
 \bar2\,1\,1\,2\,1\,\bar1 & \rightsquigarrow \begin{ytableau}\bar1 & 1 & 1 & \bar2\\ 1\\ 2\end{ytableau} \\ 
 \bar2\,1\,1\,2\,1\,\bar1\,\bar3 & \rightsquigarrow \begin{ytableau}\bar1 & 1 & 1 & \bar2\\ 1\\ 2\\ \bar3\end{ytableau} \\ 
 \bar2\,1\,1\,2\,1\,\bar1\,\bar3\,\bar1 & \rightsquigarrow \begin{ytableau} \bar1 & 1 & 1 & \bar2\\ \bar1 & \bar3\\ 1\\ 2 \end{ytableau} 
\end{align*}

 \subsection{A further illustration of Theorem \ref{mainresult}}\label{Appendfurther}
 Setting $a = 2$, $b = 2$, $c = 5$, $d =6$, $e = 3$, and $\mathcal{S} = 2$ in Lemma \ref{beforeinterpret}, the $g$-coefficient on the 
 left of the main identity in Lemma \ref{beforeinterpret} reduces to $g_{ 63 \, 2211111 \, 42111 } = 1$. In this case, the index set
 $\mathcal{J}_{d}^{+}(\nu)$ evaluates as $$ \mathcal{J}_{d}^{+}(\nu) = \big\{ (32111, 0, 2), (32111, 0, 3) \big\}, $$ and Lemma 
 \ref{beforeinterpret} then gives us that $$g_{ 63 \, 2211111 \, 42111 } = -1 + g_{62 \, 2111111 \, 32111} + g_{53 \, 2111111 \, 
 32111}. $$ The unique Blasiak tableau associated with the valuation $g_{62 \, 2111111 \, 32111} = 1$ is $$ \begin{ytableau} 
 \overline{1} & 1 & \overline{2} \\ \overline{1} & \overline{2} \\ \overline{1} \\ \overline{1} \\ 1 \end{ytableau}. $$ The unique Blasiak 
 tableau associated with the valuation $ g_{53 \, 2111111 \, 32111} = 1$ is $$ \begin{ytableau} \overline{1} & 1 & \overline{2} \\ 
 \overline{1} & \overline{2} \\ \overline{1} \\ \overline{1} \\ 2 \end{ytableau}. $$ 

\bibliographystyle{plain}
\bibliography{seprefe}

\end{document}